\documentclass[a4paper]{amsart}

\usepackage{amsmath,amsfonts,amssymb,bm}

\usepackage{cancel}
\usepackage{tikz-cd}
\usepackage{float}
\usepackage[all]{xy}

\usepackage[latin1]{inputenc}
\usepackage{xcolor}

\usepackage{soul}

\usepackage[numbers,sort&compress]{natbib}
\usepackage{graphicx}

\usepackage{hyperref}
\hypersetup{
   colorlinks   =  true,
    linkcolor    = cyan,
    citecolor    = red,
     urlcolor	=magenta,     
}

\usepackage{amsthm}

\newtheorem{definition}{Definition}[section]
\newtheorem{lemma}[definition]{Lemma}
\newtheorem{theorem}[definition]{Theorem}
\newtheorem{proposition}[definition]{Proposition}
\newtheorem{corollary}[definition]{Corollary}
\newtheorem{remark}[definition]{Remark}
\newtheorem{example}[definition]{Example}

           {\nolinebreak \hfill $\Box$ \end{trivlist}}

\def \g{\mathfrak{g}}
\def \R{\mathbb{R}}

\def \N{\mathcal{N}}

\title[Mechanical presymplectic structures and MW reduction of time-dependent Hamiltonian systems]{Mechanical presymplectic structures and Marsden-Weinstein reduction of time-dependent Hamiltonian systems}

\author{I. Gutierrez-Sagredo$^{1,2}$, D. Iglesias Ponte$^{2,3}$, J. C. Marrero$^{2,3}$, E. Padr\'on$^{2,3}$}
\date{}


\begin{document}

\maketitle

\vspace{-20pt}
\begin{center}
{\small\it  $\;^1$ Departamento de Matem\'aticas y Computaci\'on, Universidad de Burgos, 09001 Burgos, Spain}

{\small\it $\;^2$Departamento de Matem\'aticas, Estad\'{\i}stica e Investigaci\'on Operativa}\\{\small\it University of La Laguna, La Laguna, Spain}
\\[5pt]

{\small\it $\;^3$ULL-CSIC Geometr\'{\i}a Diferencial y Mec\'anica Geom\'etrica and Instituto de Matem\'aticas y Aplicaciones IMAULL-University of La Laguna, La Laguna, Spain}
\\[5pt]
e-mail: {\small \href{mailto:igsagredo@ubu.es}{igsagredo@ubu.es}, \href{mailto:idiglesia@ull.edu.es}{diglesia@ull.edu.es}, \href{mailto:jcmarrer@ull.edu.es}{jcmarrer@ull.edu.es}, \href{mailto:mepadron@ull.edu.es}{mepadron@ull.edu.es} }

\end{center}
\begin{abstract} 
In 1986, Albert proposed a Marsden-Weinstein reduction process for cosymplectic structures. In this paper, we present the limitations of this theory in the application of the reduction of symmetric time-dependent Hamiltonian systems. As a consequence, we conclude that cosymplectic geometry is not appropriate for this reduction. Motived for this fact, we replace cosymplectic structures by more general structures: mechanical presymplectic structures. Then, we develop Marsden-Weinstein reduction for this kind of structures and we apply this theory to interesting examples of time-dependent Hamiltonian systems for which Albert's reduction method doesn't work. 
\end{abstract}
\maketitle

\tableofcontents

\section{Introduction}

\subsection{Marsden-Weinstein reduction of a Hamiltonian system on  a symplectic manifold}
Symplectic manifolds play a key role in the Hamiltonian formulation of Classical Mechanics. As it 
is well known, given a configuration space $Q$ of a classical mechanical system, its cotangent bundle 
$T^*Q$, equipped with a canonical symplectic structure, is the fundamental element to describe 
Hamiltonian Mechanics. More precisely, for a Hamiltonian function $H:T^*Q\to {\mathbb R}$, Hamilton equations are just the 
integral curves of the corresponding Hamiltonian vector field (see, for instance, \cite{AM,MP, LM}).

In this context, a method that simplifies the integration of the dynamical equations of a symmetric Hamiltonian system and that  sometimes provides new examples of symplectic manifolds, is the reduction process. A particular case is the so-called Marsden-Weinstein reduction \cite{MW} (see also \cite{AM}). A Hamiltonian action of a Lie group $G$ on a symplectic manifold $(M,\Omega )$ is an action by symplectomorphisms such that there is a map $J\colon M\to \mathfrak{g}^*$, $\mathfrak{g}^*$ being the dual of the Lie algebra $\mathfrak{g}$ of $G$, 
satisfying
 \[
i_{\xi_M}\Omega=dJ_\xi,\qquad  \mbox{ for all }\xi \in {\mathfrak g},
 \]
where $\xi_M\in \mathfrak{X}(M)$ is the fundamental vector field corresponding to the Lie algebra element $\xi$ and $J_\xi\colon M\to \R$ is the real $C^\infty$-function on $M$ given by
\[
J_\xi (x)=\langle J(x),\xi \rangle, \mbox{ for }x \in M.
\]
Moreover,  $J$ is $G$-equivariant if it is equivariant with respect to the symplectic action $\phi$ and the coadjoint representation $Ad^*:G\times {\mathfrak g}^*\to {\mathfrak g}^*$. Under this hypothesis (a Hamiltonian action with $G$-equivariant momentum map), given a regular value $\mu$ of $J$, if the isotropy subgroup $G_\mu$ for the coadjoint representation acts freely and properly on $J^{-1}(\mu )$ then $M_\mu=J^{-1}(\mu )/G_\mu$ is a smooth manifold  and the Marsden-Weinstein reduction process shows that there is a naturally induced ``reduced" symplectic structure $\Omega_\mu$  on $M_\mu$. If, in addition, we have a $G$-invariant Hamiltonian function $H\colon M\to \R$ then, for each $\xi \in \mathfrak{g}$, the smooth function $J_\xi$ is a first integral of the Hamiltonian vector field $X_H$. Moreover, the Marsden-Weinstein reduction process allows to describe the reduced dynamics as the Hamiltonian dynamics with respect to a reduced Hamiltonian function 
$H_\mu\colon M_\mu\to \R$.

\subsection{Cosymplectic and symplectic reduction for time-dependent Hamiltonian systems} 
Cosymplectic structures provide an odd-dimensional counterpart to symplectic structures. A cosymplectic manifold $(M,\omega, \eta )$ is a $(2n+1)$-dimensional manifold $M$ equipped with a closed $2$-form $\omega\in\Omega^2(M)$ and a closed $1$-form  $\eta\in \Omega^1(M)$ such that $\eta\wedge \omega^n$ is a volume form on $M$. Cosymplectic structures were first introduced by Libermann \cite{Lib} and they appear, for instance, as the induced structure on the critical set of a $b$-symplectic manifold \cite{GMP}. In addition, they have been used in several geometric formulations of time-dependent Hamiltonian systems (see, for instance, \cite{ChLeMa}). 

Given a  cosymplectic structure on a manifold $M$, there is an intrinsic dynamics derived from the so-called {\it Reeb vector field }${\mathcal R}$, which is characterized by the conditions 
\[
i_{\mathcal R}\omega=0 \mbox{ and }i_{\mathcal R}\eta=1.
\]
Moreover, if  $H:M\to {\mathbb R}$ is a Hamiltonian function one can 
define the Hamiltonian vector field of $H$ characterized by 
\[
i_{X_H^{(\omega,\eta)}}\omega=dH-{\mathcal R}(H)\eta \mbox{ and } i_{X_H^{(\omega,\eta)}}\eta=0,
\]
(see for instance, \cite{CaLeLa}). 
It is remarkable that, given a Hamiltonian $H$, one can
modify $(\omega, \eta)$ and obtain a new cosymplectic structure $(\omega_H=\omega+  dH\wedge \eta,\eta)$ such that the evolution vector field $E^{(\omega,\eta)}_H=X_H^{(\omega ,\eta)}+\mathcal{R}$ of $H$ with respect to  $(\omega,\eta)$
is just the Reeb vector field ${\mathcal R}_H$  of $(\omega_H,\eta)$ (see Proposition \ref{cc}).

For time-dependent Hamiltonian systems,  and once a reference frame has been fixed, the phase space is just the product ${\mathbb R}\times T^*Q.$ The canonical coordinate $t$ in ${\mathbb R}$ represents the time and $T^*Q$ is the phase space of positions and momenta. ${\mathbb R}\times T^*Q$ admits a canonical cosymplectic structure with Reeb vector field ${\mathcal R}=\displaystyle \frac{\partial}{\partial t}$. In addition, for a time-dependent Hamiltonian function $H:{\mathbb R}\times T^*Q\to {\mathbb R}$, the integral curves of the evolution vector field $E^{(\omega,\eta)}_H$ are just the solutions of the Hamiltonian equations for $H.$

Following Marsden-Weinstein reduction, in \cite{A} C. Albert  proposed a reduction theorem for cosymplectic structures, which has been extended in the singular setting in \cite{LS}. For this purpose, Albert introduces the notion of a Hamiltonian cosymplectic action as follows. An action of a Lie group $G$ on a cosymplectic manifold $(M,\omega,\eta)$ preserving $\omega$ and $\eta$ is Hamiltonian if there is a $G$-equivariant map $J: M\to \mathfrak{g}^*$, called the momentum map, such that 
\[
 i_{\xi_M}\omega=dJ_\xi,\qquad  \mbox{ for all }\xi \in {\mathfrak g}.
\]
In this setting, there is also a Noether theorem, since 
${\mathcal R}(J_\xi)=0$, that is, the functions $J_\xi$ are first integrals of the Reeb vector field (this method has been recently used in \cite{LuMaZa}). However, Albert's reduction Theorem has several limitations. For instance, in order to induce a cosymplectic structure $(\omega _\mu ,\eta _\mu )$ on the reduced manifold $M_\mu = J^{-1}(\mu )/G_\mu,$ one needs to impose the condition
\begin{equation}\label{intro:cA}
\eta(\xi_M)=0,\qquad \mbox{ for all }\xi\in {\mathfrak g}.
\end{equation}
However, there exist Hamiltonian cosymplectic actions for interesting invariant time-dependent Hamiltonian systems which do not satisfy condition  (\ref{intro:cA}).

Even more impactful is the limitation posed by Noether's theorem. In fact, for a time-dependent Hamiltonian system on  $M={\mathbb R}\times T^*Q$ with Reeb vector field ${\mathcal R}=\displaystyle\frac{\partial}{\partial t},$ since ${\mathcal R}(J_\xi)=0, $ then conserved quantities do not depend explicitly on time. On its own, this fact is already unsatisfactory for a theory whose main interest is to deal with time-dependent Hamiltonian systems. Moreover, it goes against the spirit of any relativity theory (for instance, Galilean relativity). It would be desirable to have a theory that, assuming it is applicable to some physical system, it continues to apply to the same system but described by a different inertial observer. The simplest example would be the position of the center of mass of a $N$ particle system subject to no external force: when described by an observer moving with the system, it would be constant. However, for any other inertial observer it would depend explicitly on time (linearly on this example). This line of thought provides a virtually infinite family of examples for our theory by considering time-independent Hamiltonian systems described by different inertial observers (one such example is considered in Example \ref{example1}). It could be argued that Albert's reduction Theorem could be applied to this systems if the ``right'' observer is chosen, but as discussed before this would be against the relativity principle. Moreover, there are examples which are intrinsically different from the previous ones in which Albert's reduction Theorem cannot be applied in any case, and therefore,  a new reduction process should be proposed (see Section \ref{sec:planewave} for one such example). 

In another direction, time-dependent Hamiltonian systems may be discussed in the so-called extended formalism. In this approach, the extended phase space is a symplectic manifold: the total space of a principal $\R$-bundle on the phase space of momenta associated with the time-dependent Hamiltonian system. In addition, each symmetric time-dependent Hamiltonian system induces a symmetric Hamiltonian system on the symplectic extended phase space. So, one may apply standard Marsden-Weinstein symplectic reduction to this last system and, as a consequence, you can obtain the corresponding results for the time-dependent Hamiltonian system. These ideas were developed in \cite{LaMaPa}.

\subsection{Aim and contributions of the paper}
In this paper, we try to overcome the limitations of Albert's reduction method in the application of the reduction of symmetric time-dependent Hamiltonian systems. First of all, we 
propose a modification of the action and the moment map in order to satisfy (\ref{intro:cA}). Since this approach is not completely satisfactory (from a dynamic perspective)  because we need to know the flow of the Reeb vector field, we weaken the notion of a cosymplectic structure. More precisely, we introduce the notion of a mechanical presymplectic structure on an odd-dimensional manifold $M$ as a closed 2-form  $\omega$ of corank $1$   and a vector field ${\mathcal R}$ which  generates the characteristic foliation of $\omega$, 
\[ 
\ker\omega=\langle {\mathcal R}\rangle.
\]
Any cosymplectic structure induces a mechanical presymplectic structure but there are examples of mechanical presymplectic structures which do not come from any cosymplectic ones (see Section \ref{section4}). If we have a Lie group acting on a manifold endowed with a mechanical presymplectic structure such that 
\begin{itemize}
\item The presymplectic $2$-form $\omega$ is $G$-invariant,
\item ${\mathcal R}$ is not tangent to the $G$-orbits and 
\item There exists a smooth map $J:M\to {\mathfrak g}^*$ such that 
$$i_{\xi_M}\omega=dJ_\xi, \forall \xi\in {\mathfrak g}$$
\end{itemize}
then, under adequate topological assumptions,  there is an induced mechanical presymplectic structure on the reduced manifold $J^{-1}(\mu )/G_\mu$. This result is then used to obtain a correct reduction of the evolution dynamics of a time-dependent Hamiltonian system on a cosymplectic manifold. 

We remark that a general result on the reduction of presymplectic Hamiltonian systems with symmetry has already been considered in \cite{EMR2}. This result is compared with our approach for presymplectic structures of corank $1$.

\subsection{Organization of the paper}
The paper is organized as follows. In Section \ref{section2}, we introduce the notion of a cosymplectic structure, we recall its basic properties and  we describe two geometric formulations of time-dependent Hamiltonian dynamics using cosymplectic structures. In Section \ref{sec:3},  we present Albert's reduction procedure for cosymplectic structures and describe the problems that arise. Moreover, we present a modification 
of the action in order to solve these issues and the drawbacks.
In Section \ref{section4}, we introduce the notion of a mechanical presymplectic structure and a natural reduction procedure for this kind of structures, comparing our results with those in \cite{EMR2} (see Remark \ref{rmk:EMR2}). Combining the results of the two previous sections, we obtain a method to reduce the evolution dynamics of a cosymplectic Hamiltonian system (see Section \ref{sec:5}). In Section \ref{section:examples}, we present several examples  in order to illustrate our theory. The paper closes with an appendix which includes the proof of an algebraic result (Lemma \ref{Lema:4.5}). 

\section{Cosymplectic Hamiltonian systems and time-dependent Hamiltonian dynamics}\label{section2}

A {\it cosymplectic structure} on a $(2n+1)$-dimensional manifold $M$ is a closed $2$-form $\omega\in\Omega^2(M)$ and a closed $1$-form  $ \eta\in \Omega^1(M)$ such that $\eta\wedge \omega^n$ is a volume form on $M.$
In such a case,  we say that $(M,\omega,\eta)$ is a {\it cosymplectic manifold.}

Any cosymplectic manifold has an intrinsic dynamics derived from {\it the Reeb vector field }${\mathcal R}$, a vector field on $M$ characterized by the following conditions
\begin{equation}\label{reeb}
i_{\mathcal R}\omega=0 \mbox{ and }i_{\mathcal R}\eta=1.
\end{equation}

Note that these conditions imply that the cosymplectic structure $(\omega,\eta)$ is invariant with respect to the Reeb vector field, {\it i.e.}
\begin{equation}\label{ROE}
{\mathcal L}_{\mathcal R}\omega=0 \mbox{ and } {\mathcal L}_{\mathcal R}\eta=0.
\end{equation}

Moreover, 
$$\ker\omega=\langle {\mathcal R}\rangle$$ and we have  the following decomposition of the tangent bundle 
\[
TM=\langle \eta\rangle ^0 \oplus \langle {\mathcal R}\rangle 
\]
given by 
\[ 
X=(X-\eta(X){\mathcal R})+ \eta(X){\mathcal R},
\]
for all $X\in TM$. Here $\langle \eta\rangle ^0$ is the annihilator of $\eta$, i.e.
\[
\langle \eta\rangle ^0=\{X\in TM/\eta(X)=0\}.
\]

A cosymplectic manifold is a Poisson manifold,  where the Poisson bivector is given by 
\[
\Lambda(\alpha_1,\alpha_2)=\omega(\flat^{-1}(\alpha_1),\flat^{-1}(\alpha_2)),
\]
with $\alpha_1,\alpha_2 \in  \Omega^1(M)$ and $\flat:{\mathfrak X}(M)\to \Omega^1(M)$ the following  isomorphism of $C^\infty(M,{\mathbb R})$-modules 
\begin{equation}\label{eq:bemol}
\flat(X)=i_X\omega + \eta(X)\eta,\;\;\; X\in {\mathfrak X}(M).
\end{equation}
Cosymplectic manifolds are the odd dimensional analogues of symplectic manifolds in  the Poisson framework. As in the case of symplectic manifolds, there is a Darboux Theorem for these kind of manifolds, which states that there always exist local coordinates $(q^i, p_i,t)$ on $M$ such that the local expression of the cosymplectic structure and the associated Reeb vector field are
\begin{equation}\label{eq:local}
\omega=dq^i\wedge dp_i,\;\;\,\; \eta=dt,\;\;\;\; {\mathcal R}=\frac{\partial }{\partial t}.
\end{equation}
Now, additionally, we suppose that we have a Hamiltonian function $H:M\to {\mathbb R}.$ 
Denote by  $X_H^{(\omega,\eta)}\in {\mathfrak X}(M)$ the {\it Hamiltonian vector field of $H$ } characterized by the following conditions
\begin{equation}\label{XH}
i_{X_H^{(\omega,\eta)}}\omega=dH-{\mathcal R}(H)\eta \mbox{ and } i_{X_H^{(\omega,\eta)}}\eta=0
\end{equation}
and by $E^{(\omega,\eta)}_H$ the {\it evolution vector field }defined as the sum of $X^{(\omega,\eta)}_H$ and the Reeb vector field ${\mathcal R}$, \emph{i.e.} 
\begin{equation}
\label{eq:EH}
E^{(\omega,\eta)}_H := X^{(\omega,\eta)}_H + \mathcal{R}.
\end{equation}
 
Introducing \eqref{eq:local} and \eqref{XH} into \eqref{eq:EH}, we obtain the local expressions of the Hamiltonian vector field of $H$ and of the evolution vector field, which are
\begin{equation}\label{XHEH}
X_H^{(\omega,\eta)}=\frac{\partial H}{\partial p_i}\frac{\partial }{\partial q^i}-\frac{\partial H}{\partial q^i}\frac{\partial }{\partial p_i},\;\;\;E^{(\omega,\eta)}_H=\frac{\partial H}{\partial p_i}\frac{\partial }{\partial q^i}-\frac{\partial H}{\partial q^i}\frac{\partial }{\partial p_i}+ \frac{\partial }{\partial t}.
\end{equation}
 
Unlike for the Reeb vector field,  the cosymplectic structure $(\omega,\eta)$ is not, in general,  invariant with respect to the Hamiltonian and the evolution vector fields. In fact, we have  
\[
{\mathcal L}_{X_H^{(\omega,\eta)}}\omega={\mathcal L}_{E^{(\omega,\eta)}_H}\omega=-d({\mathcal R}(H))\wedge \eta.
\]
For more information on cosymplectic geometry and the previous dynamical vector fields see, for instance, \cite{CNY} and the references therein. The geometric structure of the cosymplectic manifold captures how the dynamics of the system evolves over time. 
 
It is important to remark that there are several geometric formulations of time-dependent Hamiltonian systems using cosymplectic structures. We will recall  two of them in what follows. The common ingredients of these formulations are: 
 \begin{itemize}
\item A fibration $\Pi:Q\to {\mathbb R}$ with total space a manifold $Q$ of dimension $n.$ 
\item The vertical vector subbundle $\Pi_{10}:V\Pi\to Q$  of the tangent bundle $\pi_Q:TQ\to Q$ induced by $\Pi$ whose fiber at $q\in Q$ is 
$$V_q\Pi=\{v_q\in T_q Q/ T_q\Pi(v_q)=0\}.$$
We denote by $\Pi_{10}^*:V^*\Pi\to Q$ the dual vector bundle of $\Pi_{10}$ and by $\Pi_{1}^*:V^*\Pi\to {\mathbb R}$ the composition $\Pi_{1}^*=\Pi\circ \Pi_{10}^*.$
\item The principal ${\mathbb R}$-bundle $\nu:T^*Q\to V^*\Pi$ dual of the inclusion $\iota:V\Pi\to TQ$. Note that the principal action of the additive group
$\R$ on $T^*Q$ is given by $t\cdot \alpha _q=\alpha _q+t\Pi ^*(dt)(q)$, for $t\in 
\R$ and $\alpha _q\in T^*_qQ$. 
\item A Hamiltonian section (a smooth section of $\nu$) $h:V^*\Pi\to T^*Q.$
\end{itemize}
The following diagram illustrates the previous situation
\[
\xymatrix{&T^*Q\ar[dr]_{{\nu}}\ar[dl]_{\pi_Q}&&\\
Q  \ar[d]^{\Pi}&&V^*\Pi\ar[dll]^{\Pi_1^*}\ar@/_/[ul]_h\ar[ll]_{\Pi_{10}^*}
&\\
{\mathbb R}&&&
}
\]
The local expressions of these elements are the following. Since  $\Pi:Q\to {\mathbb R}$ is a submersion then  one may consider local coordinates $(q^i,t)$ on $Q$ adapted to $\Pi$ such that the local expression of $\Pi$ is $\Pi(q^i,t)=t$. Denote by   $(q^i,t,p_i,p)$ (resp., $(q^i,t,p_i)$) the 
corresponding local coordinates on $T^*Q$ (resp., on $V^*\Pi$). With respect to them, we have that   
\[
\Pi_{10}^*(q^i,t,p_i)=(q^i,t),\;\;\; \Pi_1^*(q^i,t,p_i)=t \mbox{ and } \nu(q^i,t,p_i,p)=(q^i,t,p_i).
\]
Moreover, the  Hamiltonian section  $h:V^*\Pi\to T^*Q$ is given by 
\begin{equation}\label{H}
 h(q^i,t,p_i)=(q^i,t,p_i,-H_h(q^i,t,p_i)).
\end{equation}
with $H_h$  a local function on $V^*\Pi.$

With these elements the two different formulations on the restricted phase space of momenta $V^*\Pi$ are given as follows. 

\begin{itemize}
 \item[a)]{\it \sc Reeb cosymplectic formalism \cite{ChLeMa,LMM}.}
 In this first formalism,  the cosymplectic structure on $V^*\Pi$ is 
 \begin{equation}\label{Oh}
 (\omega_h=h^*\omega_Q, \eta=(\Pi_1^*)^*(dt)),
 \end{equation}
where $\omega_Q$ is the canonical symplectic structure on $T^*Q$. Then, the 
integral curves of the corresponding Reeb vector field ${\mathcal R}_h$ are the 
solutions of {Hamilton equations} for $h$. 

The local expression of the cosymplectic structure is
\begin{equation}\label{wh}
 \omega_h = d q^i\wedge dp_i + \frac{\partial H_h}{\partial q^i}d q^i\wedge d t +\frac{\partial{H_h}}{\partial p_i}d p_i\wedge d t, \;\;\;\;\; \eta=d t.
\end{equation}
Thus,  the Reeb vector field  ${\mathcal R}_h$ of the cosymplectic structure $(\omega_h,\eta)$ has  the following  local expression
\begin{equation}\label{R}
{\mathcal R}_h  = \frac{\partial }{\partial t} +  \frac{\partial H_h}{\partial p_i} \frac{\partial }{\partial q^i} -  \frac{\partial H_h}{\partial q^i} \frac{\partial }{\partial p_i}, 
\end{equation}
and the integral curves of ${\mathcal R}_h$ are the solutions of \emph{the Hamilton equations} for $h$ given by  
\begin{equation}\label{HE}
\dot{t}=1,\quad\quad \frac{dq^i}{dt} = \frac{\partial H_h}{\partial p_i}, \quad\quad \frac{dp_i}{dt} = - \frac{\partial H_h}{\partial q^i}, \qquad \mbox{for all } i \in \{1,\ldots,n\}.
\end{equation}

The following diagram summarizes the formalism 
$$
\xymatrix{&(T^*Q,\omega_Q)\ar[dr]_{\nu}\ar[dl]_{\pi_Q}&&\\
Q  \ar[d]^{\Pi}&&(V^*\Pi,{(\omega_h,\eta}),{\mathcal R}_h)\ar[dll]^{\Pi_1^*}\ar@/_/[ul]_h\ar[ll]_{\Pi_{10}^*}
&\\
{\mathbb R}&&&}
$$

 \item[b)]{\sc Evolution cosymplectic formalism.}  Some ideas on this formalism are contained in \cite{EMR}. In this case,  we consider  an  Ehresmann connection of $\Pi$, that is,   a vector field 
 $Y\in {\mathfrak X}(Q)$ such that  
 \begin{equation}\label{Er}
 \Pi^*(dt)({Y})=1.
 \end{equation}
 
Now, we define the  $1$-form $\lambda_Y$ on $V^*\Pi$ imitating the construction of the standard Liouville $1$-form on $T^*Q$
$$\lambda_Y(\alpha)(X_\alpha)=\alpha(T_\alpha\Pi_{10}^*(X_\alpha)-\eta(T_\alpha\Pi_{10}^*(X_\alpha))Y(\Pi_{10}^*(\alpha)))$$
for all  $\alpha\in V^*\Pi$ and $X_\alpha\in T_\alpha(V^*\Pi).$
Here $\eta$ is the $1$-form on $V^*\Pi$ defined in (\ref{Oh}). 

Note that $\lambda _Y$ is well defined. In fact, from (\ref{Er}),
we have that $T_\alpha\Pi_{10}^*(X_\alpha)-\eta(T_\alpha\Pi_{10}^*(X_\alpha))Y(\Pi_{10}^*(\alpha))\in V\Pi.$

If  the local expression of the vector field $Y$ is 
\[
Y(q,t)=Y^i(q,t)\frac{\partial }{\partial q^i} + \frac{\partial }{\partial t},
\] 
with $Y^i$  functions on the corresponding coordinated neighborhood on $Q,$ the local expressions of the $1$-forms $\lambda_Y$ and $\eta$  are  
\[
\lambda_Y=p_idq^i-Y^ip_idt,\mbox{ and } \eta=dt.
\]

Consider the $2$-form $\omega_Y:=-d\lambda_Y$ on $V^*\Pi$ whose local expression is 
\begin{equation}\label{wY}
\omega_Y=dq^i\wedge dp_i + (\frac{\partial Y^i}{\partial q^j}p_idq^j + Y^idp_i)\wedge dt.
\end{equation} 
Then $(\omega_Y,\eta)$ is a cosymplectic structure on $V^*\Pi$ with Reeb vector field $R_Y,$ whose local expression is 
$$R_Y=\frac{\partial}{\partial t} + Y^i\frac{\partial }{\partial q^i}-\frac{\partial Y^i}{\partial q^j}p_i\frac{\partial }{\partial p_j}$$

Now, let $h:V^*\Pi\to T^*Q$ be a Hamiltonian section of $\nu:T^*Q\to V^*\Pi$ with local expression 
$$h(q^i,p_i,t)=(q^i,p_i,t,-H_h(q^i,p_i,t))$$
with $H_h$ a local function on $V^*\Pi.$ 

Let $H_h^Y$ be the function on $V^*\Pi$ defined by 
$$H_h^Y=-Y^l\circ h$$ 
where $Y^l:T^*Q\to {\mathbb R}$ is the fiberwise linear function induced by $Y$, i.e.
\[
Y^l(\alpha_q)=\alpha_q(Y(q)), \mbox{ for all } q\in Q \mbox{ and } \alpha_q\in T^*_qQ.
\]
The local expression of $H_h^Y$ is 
\[H_h^Y(q^i,p_i,t)=-Y^ip_i + H_h(q^i,p_i,t).\]

The Hamiltonian vector field $X_{H_h^Y}^{(\omega_Y,\eta)}$ of this function with respect to the cosymplectic structure $(\omega_Y,\eta)$ is given locally by  
$$X_{H_h^Y}^{(\omega_Y,\eta)}=(\frac{\partial H_h}{\partial p_i}-Y^i)\frac{\partial }{\partial q^i}-( \frac{\partial H_h}{\partial q^i}-\frac{\partial Y^j}{\partial q_i}p_j)\frac{\partial }{\partial p_i}.$$

In this case, the local expression of the evolution vector field $E^{(\omega,\eta)}_{H_h^Y}$ is 
$$E^{(\omega,\eta)}_{H_h^Y} = \frac{\partial }{\partial t} +  \frac{\partial H_h}{\partial p_i} \frac{\partial }{\partial q^i} -  \frac{\partial H_h}{\partial q^i} \frac{\partial }{\partial p_i}.$$
So, $E^{(\omega,\eta)}_{H^Y_h}$ is just the Reeb vector field ${\mathcal R}_h$ (see (\ref{R})) of the cosymplectic structure $(\omega_h,\eta)$ defined in the Reeb cosymplectic formalism described previously. 

 The following diagram shows the ingredients of this formalism 
\[  
\xymatrix{&(T^*Q,\lambda_Q)\ar[dr]_{p_1}\ar[r]^{\hspace{1cm}-Y^l}\ar[dl]_{\pi_Q}&{\mathbb R}&\\
Q  \ar[d]^{\Pi}&&(V^*\Pi,{(\omega_Y=-d{\lambda_Y},(\Pi_1^*)^*(dt)}),E^{(\omega,\eta)}_{H_h^Y})\ar[u]_{H^Y_h}\ar[dll]^{\Pi_1^*}\ar@/_/[ul]_h\ar[ll]_{\Pi_{10}^*}
&\\
{\mathbb R}&&&
}
\]
\end{itemize}
Moreover, from (\ref{wh}) and (\ref{wY}), we conclude that the $2$-forms $\omega_h$ and $\omega_Y$  of the cosymplectic  structures of both formalisms  are related by the formula 
\[
\omega_h= \omega_Y+  dH_h^Y\wedge (\Pi_1^*)^*(dt).
\]
In fact, a straightforward computation proves that every  cosymplectic dynamics of a Hamiltonian system determined by its evolution vector field can be seen as a Reeb dynamics associated to other cosymplectic structure on the same space (see, for instance, \cite{JoLu}).
\begin{proposition} \label{cc}
Let $(M,\omega,\eta)$ be a cosymplectic manifold and $H:M\to {\mathbb R}$ be a Hamiltonian function. Then, the pair $(\omega_H,\eta)\in \Omega^2(M)\times \Omega^1(M)$, with
\begin{equation}\label{wH}
\omega_H= \omega+  dH\wedge \eta,
\end{equation}
is a new cosymplectic structure on $M$ such that the evolution vector field $E^{(\omega,\eta)}_H$ of $H$ with respect to  $(\omega,\eta)$
is just the Reeb vector field ${\mathcal R}_H$  of $(\omega_H,\eta).$
\end{proposition}

Therefore, the evolution dynamics is a Reeb dynamics and this fact is a motivation to use some results  on Reeb dynamics in the cosymplectic framework, for instance, reduction processes. In particular, in \cite{A} it is given a version of Marsden-Weinstein reduction for a cosymplectic manifold. In the following sections, we will  show the limitations for applying this reduction process to the Reeb and the evolution dynamics and how overcome them. 

\section{The limitations of Albert's reduction for  discussing  the Reeb and the evolution dynamics}
\label{sec:3}

\subsection{Albert's Theorem}
Firstly, we recall Albert's reduction theorem for cosymplectic manifolds.  

\begin{definition} Let $(M,\omega,\eta)$ be a cosymplectic manifold.  
An action $\phi:G\times M\to M$ is cosymplectic if 
\[
\phi_g^*(\omega)=\omega\mbox{ and } \phi_g^*(\eta)=\eta.\]
\end{definition}
In this case, we have that 
\begin{equation}\label{RO}
0={\mathcal L}_{\xi_M}\omega=di_{\xi_M}\omega,\qquad \;0={\mathcal L}_{\xi_M}\eta=d(\eta(\xi_M)),\end{equation}
where $\xi_M$ is the infinitesimal generator of $\phi$ associated with $\xi\in{\mathfrak g}$. Therefore,  $i_{\xi_M}\omega$ is a closed $1$-form and, if $M$ is connected, then $\eta(\xi_M)$ is constant. 

Moreover, one can prove that the Reeb vector field ${\mathcal R}$ of $(M,\omega,\eta)$ is $G$-invariant, that is,  
\[
T_x\phi_g({\mathcal R}(x))={\mathcal R}(\phi _g(x)),
\]
for all $g\in G$ and $x\in M.$ 

The following result shows that the presence of a cosymplectic action on a connected cosymplectic manifold implies the existence of a distinguished $1$-cocycle on the corresponding Lie algebra associated with the Lie group. 
\begin{proposition}
If an action $\phi:G\times M\to M$ is cosymplectic on a connected cosymplectic manifold $(M,\omega,\eta)$,  then the element $c_\eta:{\mathfrak g}\to \R$ of ${\mathfrak g}^*$ given by 
\begin{equation}\label{cocycle}
c_\eta(\xi)=\eta(\xi_M), \mbox{ for all }\xi\in {\mathfrak g},
\end{equation}
is a cocycle of ${\mathfrak g}$. As a consequence, the corresponding left invariant $1$-form $\overleftarrow{c_\eta}$ on $G$ is closed. 
\end{proposition}
\begin{proof}
We have that, for all $\xi,\xi'\in{\mathfrak g},$
$$
c_\eta([\xi,\xi']_{\mathfrak g})=\eta(([\xi,\xi']_{\mathfrak g})_M)=-\eta([\xi_M,\xi'_M])=d\eta(\xi_M,\xi'_M)-\xi_M(\eta(\xi'_M))+\xi'_M(\eta(\xi_M))=0.$$

Here we have used that $\eta$ is a closed 1-form, $\eta(\xi _M )$ and $\eta (\xi '_M)$ are constant functions and the equality 
\[
([\xi ,\xi ']_{\mathfrak{g}})_M = - [\xi _M ,\xi '_M].
\]
\end{proof}
\begin{definition}
A cosymplectic action $\phi:G\times M\to M$ is Hamiltonian if there is a smooth map $J: M\to \mathfrak{g}^*$, called the momentum map, such that 
 \begin{equation}\label{Je}
 i_{\xi_M}\omega=dJ_\xi,\qquad  \mbox{ for all }\xi \in {\mathfrak g},
 \end{equation} 
where $J_\xi\colon M\to \R$ is the real $C^\infty$-function on $M$ given by
\[
J_\xi (x)=\langle J(x),\xi \rangle, \mbox{ for }x \in M.
\]
Moreover,  $J$ is $G$-equivariant if it is equivariant with respect to the action $\phi$ and the coadjoint representation $Ad^*:G\times {\mathfrak g}^*\to {\mathfrak g}^*$, that is, 
\[
J\circ \phi_g=Ad^*_g\circ J, \qquad \mbox{ for all } g\in G.
\]
\end{definition} 

Note that, in such a case, we have that ${\mathcal R}(J_\xi)=i_{\xi_M}i_{\mathcal R}\omega=0$, that is, for a Hamiltonian cosymplectic action with momentum map $J$, the functions $J_\xi$ are first integrals of the Reeb vector field.  This is a version, in the cosymplectic realm, of Noether theorem for Hamiltonian actions on symplectic manifolds (see, for instance, \cite{AM}).

In \cite{A} Albert proves the following reduction theorem for certain Hamiltonian cosymplectic actions. 
\begin{theorem}
 Let $\phi:G\times M\to M$ be a Hamiltonian cosymplectic action of a Lie group $G$ on the cosymplectic manifold  $(M,\omega,\eta)$   with $G$-equivariant momentum map $J:M\to {\mathfrak g}^*$ such that 
 \begin{equation}\label{cA}
 \eta(\xi_M)=0,\qquad \mbox{ for all }\xi\in {\mathfrak g}.
 \end{equation}
For a regular value $\mu\in {\mathfrak g}^*$ of $J$, suppose that the space of orbits $M_\mu=J^{-1}(\mu)/G_\mu$ is a manifold such that the canonical projection $\pi_\mu:J^{-1}(\mu)\to M_\mu$ is a submersion. Then,
 \begin{enumerate}
 \item[(i)] There exists a unique cosymplectic structure  $(\omega_\mu,\eta_\mu)$ on $M_\mu$ such that 
 \[
\pi_\mu^*(\omega_\mu)=\iota_\mu^*\omega\mbox{ and }\pi_\mu^*(\eta_\mu)=\iota_\mu^*\eta,
 \]
 where $\iota_\mu\colon J^{-1}(\mu )\to M$ is the canonical inclusion.
\item[(ii)] The vector field ${\mathcal R}$ restricts to $J^{-1}(\mu)$ and its restriction is $\pi_\mu$-projectable. Its projection is the Reeb vector field of  $(M_\mu,\omega_\mu,\eta_\mu).$
 \end{enumerate}
 \end{theorem}

\begin{remark}
 In \cite{A}, Albert  omits the condition ``$J$ is $G$-equivariant'' because one can change the coadjoint action $Ad^*:G\times {\mathfrak g}^*\to {\mathfrak g}^*$ for 
the affine representation 
\[
Af^\sigma: G\times {\mathfrak g}^*\to {\mathfrak g}^*,\qquad  Af^\sigma_g(\mu)=Ad_g^*\mu + \sigma(g),
\]
where $\sigma: G\to {\mathfrak g}^*$ is the cocycle of the cohomology of $G$ with values in ${\mathfrak g}^*$ given by 
$$\sigma(g)=J(\phi_g(x))-Ad^*_gJ(x), \qquad \mbox{ for all }x\in M.$$
This cocycle doesn't  depend of the chosen point $x. $ So, the map $J$ is $G$-equivariant with respect to $\phi$ and  $Af^\sigma.$ Moreover, i) and ii) from the previous Theorem are satisfied when changing $M_\mu$ by the space $J^{-1}(\mu)/\tilde G_\mu$, where $\tilde G_\mu$ is the isotropy subgroup of $\mu$ with respect to $Af^\sigma$. For more details see \cite{LuMaZa}.

\label{rem:aff}
\end{remark}

\subsection{The limitations of Albert's Theorem for Reeb dynamics}

The first comment about Albert's  reduction process is that there exist  several interesting time-dependent examples where the condition (\ref{cA}) doesn't work. In general, if $M$ is connected, we have only that $\eta(\xi_M)=constant,$ for each $\xi\in {\mathfrak g},$ but not necessarily zero. In such a case, Albert's Theorem doesn't work.

In what follows, we will show a simple example in which this situation holds. 

\begin{example}\label{Ex:N-Albert}
    We consider the standard cosymplectic structure on $M =  T^*\R^{N}\times  \mathbb R \simeq \mathbb R^{2N+1}$ given by  $$\omega = d q^i \wedge d p_i\mbox{ and }\eta = d t,$$
and the Lie group action $\phi: \R \times M \to M$,  defined by 
\begin{equation*}
    \phi (s, (q^i,p_i,t)) = (q^1-v s,q^\alpha,p_i,t + s) .
\end{equation*}
Hereafter, the Latin indexes run from $1$ to $N$ while Greek indexes run from $2$ to $N$. 

Clearly, the cosymplectic structure $(\omega,\eta)$ is invariant with respect this action, and therefore, $\phi$ is a cosymplectic action. The fundamental vector fields associated to this action are given by
\begin{equation}\label{xiM}
    \xi_M =  \xi(\frac{\partial}{\partial t}-v \frac{\partial}{\partial q^1}), \mbox{ for all } \xi\in \R.
\end{equation}
Note that $\eta (\xi_M) = \xi$ and, in consequence, this action doesn't satisfy \eqref{cA}, that is, Albert's theorem cannot be applied.  
\end{example}

The condition \eqref{cA} is necessary in Albert's Theorem  in order to obtain the reduced $1$-form $\eta_\mu$ on $M_\mu =J^{-1}(\mu )/G_\mu$. In fact, in order to assure the existence of the 1-form $\eta _\mu$ on $M_\mu$, we must guarantee that $i^*_\mu\eta$ is a $\pi_\mu$-basic 1-form on $J^{-1}(\mu )$. This, in particular, implies that
\[
\langle i^*_\mu \eta ,\xi _{J^{-1}(\mu)}\rangle =0, \quad \mbox{ for all } \xi \in \mathfrak{g}_\mu.
\] 
A first strategy for overcoming this limitation can be to modify the action and the momentum map in such a way that, with these new elements, Albert's conditions hold. In fact, this is possible, under certain integrability conditions.
\begin{proposition}\label{ns}
Let $\phi:G\times (M, \omega,\eta)\to (M, \omega,\eta)$ be a  Hamiltonian cosymplectic action on the connected cosymplectic manifold $ (M, \omega,\eta)$ with $G$-equivariant momentum map $J:M\to  {\mathfrak g}^*$. 

If  the Reeb vector field ${\mathcal R}$ associated with the cosymplectic structure $(\omega,\eta)$  is complete with flow
$\Phi^{\mathcal R}:{\mathbb R}\times M\to M$
and the left-invariant $1$-form $\overleftarrow{c_\eta}$ induced by the cocycle $c_\eta,$ defined in (\ref{cocycle}), is exact, then there is a unique multiplicative function $f_\eta:G\to {\mathbb R}$ such that $\overleftarrow{c_\eta}=df_\eta$ and the map $\widetilde{\phi}:G\times M\to M$ given by 
\begin{equation}\label{mac}
\widetilde\phi_g(x)=\Phi^{\mathcal R}_{f_\eta(g^{-1})}(\phi_g(x)),\qquad \mbox{ for all } g\in G \mbox{ and }x\in M,
\end{equation}
defines a Hamiltonian cosymplectic action with $G$-equivariant momentum map ${J}:M\to {\mathfrak g}^*$ such that 
\[\eta(\widetilde\xi_M)=0,\]
where $\widetilde{\xi}_M$ is the infinitesimal generator of $\xi\in {\mathfrak g},$ with respect to $\widetilde{\phi}$. 
\end{proposition}

\begin{proof}
The fact that $\phi$ is a Hamiltonian cosymplectic action implies that $\mathcal{R}$ is invariant under the action $\phi$, or equivalently, $[\xi_M,{\mathcal R}]=0$, for all $\xi\in {\mathfrak g}.$ Therefore, if $g\in G,$ the flow $\Phi ^{\mathcal{R}}_t$ of $\mathcal{R}$ commutes with $\phi _g$. Using this relation and the fact that $f_\eta$ is multiplicative, it is easy to prove that $\widetilde{\phi}$ is an action of $G$ on $M$. 

Next, we will see that 
\[
\widetilde{\xi}_M=\xi_M- \eta(\xi_M) \mathcal{R}, \mbox{ for }\xi\in {\mathfrak g}.
\]
In fact, we have that 
\[
\widetilde{\xi}_M(x)=(T_e\widetilde{\phi}_x)(\xi ), \mbox{ for }x\in M,
\]
where $\widetilde{\phi}_x \colon G\to M$ is the map given by 
\[
\widetilde{\phi}_x(g)=\widetilde{\phi}_g(x)=\Phi _{f_\eta (g^{-1})} ^{\mathcal{R}} (\phi _g(x)), \mbox{ for }g\in G.
\]
So, $\widetilde{\phi}_x$ may be seen as the following composition 
\[
\xymatrix{
\widetilde{\phi}_x\colon \ar[rr]^{(f_\eta \circ i, \phi _x)} G&& M\times {\mathbb R} \ar[rr]^{\Phi ^{\mathcal{R}}} && M
}
\]
where $i\colon G\to G$ is the inversion in $G$. This implies that 
\[
\widetilde{\xi}_M=(T_{(0,x)}\Phi ^\mathcal{R} )(-c_\eta (\xi )\frac{\partial}{\partial t}_{|t=0}, \xi _M (x))
\]
where, in the last equality, we have used that 
\[
df_\eta (e)=c_\eta, \qquad T_ei=-Id,
\]
with $e\in G$ the identity element in $G$. Thus, since $\Phi ^\mathcal{R}$  is the flow of $\mathcal{R}$, we conclude that
\[
\widetilde{\xi}_M(x)=\xi_M(x)-c_\eta (\xi )\mathcal{R}(x)=
\xi_M(x)-\eta (\xi _M)(x) \mathcal{R}(x).
\]
Now, from \eqref{mac} and, using that the action $\phi \colon G\times M\to M$ is cosymplectic and that $\mathcal{R}$ preserves $\omega$ and $\eta$, it follows that
\[
\widetilde{\phi}_g^\ast \omega = \phi _g^* ((\Phi ^\mathcal{R}_{f_\eta (g^{-1})} )^* (\omega ))=\omega,\qquad  \widetilde{\phi}_g^\ast \eta = \phi _g^* ((\Phi ^\mathcal{R}_{f_\eta (g^{-1})} )^* (\eta ))=\eta.
\]
So, $\widetilde{\phi}\colon G\times M\to M$ is a cosymplectic action.

On the other hand, if $\xi\in \g,$
\[
i_{\widetilde\xi_M }\omega =i_{\xi_M}\omega- \eta(\xi_M) i_\mathcal{R}\omega =dJ_\xi,
\]
which means that the action $\widetilde{\phi}$ is Hamiltonian with momentum map $J\colon M\to \mathfrak{g}^*$. 

Finally, since $\eta(\mathcal{R})=1$, then 
\[
\eta(\widetilde\xi_M)=\eta(\xi_M)-\eta(\xi_M) \eta({\mathcal R})=0.
\]
\end{proof}
However, if we are interested in using the reduction of $(M,\omega,\eta)$ with the action $\widetilde\phi$ to discuss the reduction of Reeb dynamics, the problem is that we need to know the flow of the Reeb vector field before making the reduction. So, Proposition \ref{ns}, although it may be useful at a theoretical level, is not useful for reducing to Reeb dynamics. 

\subsection{ The limitations of Albert's Theorem for evolution dynamics}

Now, suppose that $H:M\to {\mathbb R}$ is a $G$-invariant Hamiltonian function, i.e.
$$H\circ \phi_g=H.$$

In what follows, we will look at the problems that arise when we try to apply Albert's reduction Theorem to the evolution dynamics induced by $H$.

If our initial cosymplectic structure $(\omega,\eta)$ satisfies the condition (\ref{cA}) of Albert's Theorem, from (\ref{Je}),  the condition (\ref{cA}) and the invariance of $H$, we deduce that 
$$X_H^{(\omega,\eta)}(J_\xi)=i_{X_H^{(\omega,\eta)}}i_{\xi_M}\omega=-\xi_M(H)+{\mathcal R}(H)\eta(\xi_M)=0,$$
which implies that the functions $J_\xi$ are first integrals of $X_H^{(\omega,\eta)}$. On the other hand, since the action is Hamiltonian, then  functions $J_\xi$ are first integrals of ${\mathcal R}$.  Therefore, they are first integrals  of $E^{(\omega ,\eta )}_H$. 

From the dynamical point of view, the functions $J_\xi$ should be  first integrals of $E^{(\omega ,\eta )}_H$, but  they do not need to be first integrals of both vector fields ${\mathcal R}$ and ${X}_H^{(\omega,\eta)}.$ 

The following example (related with Example \ref{Ex:N-Albert}) illustrates a physical system where the momentum map defines  first integrals of $E^{(\omega,\eta)}_H$ but not of ${\mathcal R}$ and $X_H^{(\omega,\eta)}.$

\begin{example}\label{example1}{\bf $N$-dimensional harmonic oscillator seen by an observer that moves with a constant velocity.}
{
Consider a classical particle of mass $m$ moving in a harmonic oscillator potential with frequency $\Omega$ in $N$ dimensions. The origin of the harmonic potential is assumed to be $\mathbf 0 = (0,\ldots,0) \in \mathbb R^N$. Let us suppose that this system is described by an inertial observer that moves with constant velocity $\mathbf{v} = (v_1,\ldots,v_N)$. Without any lack of generality, we can take $\mathbf{v} = (v,0,\ldots,0)$, and thus the Hamiltonian function of this system is given by
\begin{equation*}
    H(q^i, p_i, t) = \frac{1}{2m} \sum_{i=1}^N p_i^2 + \frac{1}{2} m \Omega^2 \bigg( (q^1+v t)^2 + \sum_{\alpha=2}^N (q^\alpha)^2 \bigg) .
\end{equation*}
Here, the Latin indexes run from $1$ to $N$ while Greek indexes run from $2$ to $N$. 
In this case  
\begin{equation*}
    dH=\frac{1}{m} \sum_{i=1}^N p_i d p_i + m \Omega^2 \bigg( (q^1+v t) d q^1 + \sum_{\alpha=2}^N q^\alpha d q^\alpha + (q^1 + v t) v dt \bigg).
\end{equation*}

We consider the standard cosymplectic structure $(\omega,\eta)$  on $M =  T^*\R^{N}\times \mathbb R \simeq \mathbb R^{2N+1}$ given in Example \ref{Ex:N-Albert}.

The expression of the Hamiltonian vector field and the evolution vector field are  

\begin{equation*}
    X^{(\omega,\eta)}_H=\frac{1}{m} \sum_{i=1}^N p_i \frac{\partial}{\partial q^i} - m \Omega^2 \bigg( v t \frac{\partial}{\partial p_1} + \sum_{i=1}^N q^i \frac{\partial}{\partial p_i}\bigg)
\end{equation*}
and 
\begin{equation}\label{E}
    E^{(\omega,\eta)}_H=\frac{1}{m} \sum_{i=1}^N p_i \frac{\partial}{\partial q^i} - m \Omega^2 \bigg( v t \frac{\partial}{\partial p_1} + \sum_{i=1}^N q^i \frac{\partial}{\partial p_i}\bigg) + \frac{\partial}{\partial t}. 
\end{equation}

For this system, one has the following symmetry   $\phi: \R \times M \to M$,  defined  by 
\begin{equation*}
    \phi (s, q^i,p_i,t) = (q^1-v s,q^\alpha,p_i,t + s) .
\end{equation*}

Now, we want to apply Reeb  reduction to the dynamical system $(T^*\R^N\times \R, E^{(\omega,\eta)}_H)$ for the symmetry $\phi$. In order to do this, we modify the standard cosymplectic structure $(\omega,\eta)$ (see Example \ref{Ex:N-Albert}) following   Proposition \ref{cc}. The new one  is given by 
\begin{equation}\label{w}
    \omega_H = \omega + dH \wedge \eta = \sum_{i=1}^N d q^i \wedge d p_i + \frac{1}{m} \sum_{i=1}^N p_i d p_i \wedge d t + m \Omega^2 (\sum_{i=1}^N q^i d q^i \wedge dt + v t d q^1 \wedge d t),\;\;\; \eta=dt.
\end{equation}
The vector field $E^{(\omega,\eta)}_H$ is the Reeb vector field of this new cosymplectic structure. 

 On the other hand, since  $H$  and the cosymplectic structure $(\omega,\eta)$ are invariant with respect to $\phi$, then this action is cosymplectic with respect to $(\omega_H,\eta)$. 
 
 In this case, for all $\xi\in \R$, the fundamental vector field $\xi_M$ associated to this action 
satisfies  $\eta (\xi_M) = \xi$ and thus  Albert's condition doesn't work (see Example \ref{Ex:N-Albert}).
Integrating the fundamental vector field $1_M$ associated with $1\in \R,$ we get
\begin{equation}
    \begin{split}
        q^1(s) &= -v s + q^1 , \\
        q^\alpha (s) &= q^\alpha , \\
        p_i(s) &= p_i, \\
        t(s) &= s + t .
    \end{split}
    \label{eq:}
\end{equation}
A momentum map $J\colon  T^*\R^N\times \R\to \R $ for $(\omega, \eta)$ satisfies that $i_{\xi_M} \omega = - \xi v d p_1 = d J_\xi$, so we can take  $J_\xi(q^i,p_i,t) = - \xi v p_1$.

For the cosymplectic structure $(\omega _H,\eta)$, we have that the following map $J_H: T^*\R^N\times \R\to \R$ given by 
\begin{equation}
\label{eq:J_Ha}
        J_H
        =J-H=- v p_1 -  \displaystyle\frac{1}{2m} \bigg( (p_1 + m v)^2 + \displaystyle\sum_{\alpha=1}^N p_\alpha^2 \bigg) - \displaystyle\frac{1}{2} m \Omega^2 \bigg( (q^1+v t)^2 + \displaystyle\sum_{\alpha=2}^N (q^\alpha)^2 \bigg)    
\end{equation}
is an equivariant momentum map. In fact,
\[
i_{\xi _M}\omega _H = d ( \xi (J_1 - H)) = d (J_H)_\xi, \mbox{ for }\xi \in \R.
\]
Moreover, since $E_H^{(\omega ,\eta )}$ is the Reeb vector field of $(\omega _H,\eta )$, it is clear that
\[
E^{(\omega ,\eta)}_H ((J_H)_\xi )=d(J_H)_\xi (E^{(\omega ,\eta)}_H)= (i_{\xi _M}\omega _H )(E^{(\omega ,\eta)}_H)=0.
\]
So, $(J_H)_\xi$ is a first integral of $E^{(\omega ,\eta)}_H$. However, 
\[
{X}_H^{(\omega,\eta)}(J_H)=m\Omega^2v(q^1+ vt)=-{\mathcal R}(J_H),
\]
where ${\mathcal R}=\displaystyle\frac{\partial}{\partial t}$ is the Reeb vector field of $(\omega,\eta).$
}
\end{example}

Summarizing, Albert's reduction Theorem  has several limitations: 
\begin{enumerate}
\item 
There exist Hamiltonian cosymplectic actions for invariant time-dependent Hamiltonian systems
which do not satisfy condition  (\ref{cA}). However,  this condition is necessary in order to obtain a reduced cosymplectic structure. This fact suggests that, if we want to reduce these other systems, the  
cosymplectic reduction is not  the right framework. 
\item The appropriate modification of the action (and of the momentum map)  to obtain the conditions of Albert is not useful to reduce the Reeb dynamics  because we need to know the flow of the Reeb vector field. 
\item 
The conditions in Albert's Theorem imply that the functions $J_\xi$ are first integrals of  the Reeb and Hamiltonian vector fields. But, for our purpose, it would be enough that they were only first integrals of 
the evolution vector field of the Hamiltonian function. 
\end{enumerate}

\section{Reduction of mechanical presymplectic manifolds}\label{section4}

In this  section, we will see how presymplectic structures of corank 1 with parallelizable characteristic foliation provide an appropriate setting for overcoming the previous limitations.

\begin{definition}
Let $M$ be a manifold of dimension $(2n+1)$. 
A mechanical presymplectic structure on $M$ is a couple $(\omega, {\mathcal R})$, where $\omega$ is a closed 2-form of corank 1 and ${\mathcal R}$ is a vector field which generates the characteristic foliation of $\omega$, that is, 
\[ 
\ker\omega=\langle {\mathcal R}\rangle.
\]
${\mathcal R}$ is called the Reeb vector field of the mechanical presymplectic structure. The triple $(M,\omega, \mathcal{R})$ is called a mechanical presymplectic manifold. 
\end{definition}
It is clear that cosymplectic structures on a manifold $M$ of odd dimension are mechanical presymplectic structures on $M$. However, there exist mechanical presymplectic structures on manifolds of odd dimension which do not admit cosymplectic structures as we will see below. 

First of all, we remark that the Betti numbers of a compact cosymplectic manifold are all non-zero. In fact, let $(\omega ,\eta)$ be a cosymplectic structure on a compact manifold $M$ of dimension $2n+1$. Then, using Stoke's Theorem and the fact that $\eta \wedge \omega ^n$ is a volume form, one may prove that 
\[
0\neq [\eta\wedge \omega^k]\in H^{2k+1}_{DR}(M), \qquad 0\neq [\omega ^k]\in H^{2k}_{DR}(M),\mbox{ for } k=0,\ldots , n,
\]
where $H^*_{DR}(M)$ is the De Rham cohomology of $M$.

Secondly,  two cosymplectic structures $(\omega,\eta)$ and $(\omega',\eta')$ on a manifold $M$ induce the same mechanical presymplectic structure if and only if $\omega=\omega'$, both Reeb vector fields, ${\mathcal R}$ and  ${\mathcal R}',$ coincide  and $\eta'=\eta + \gamma$, with $\gamma$ a closed $1$-form on $M$ which belongs to the annihilator  $\langle{\mathcal R}\rangle^0$ of the Reeb vector field ${\mathcal R}$ of both cosymplectic structures. In fact, if $(\omega,\eta)$ and $(\omega',\eta')$ induce the same mechanical presymplectic structure then $\omega=\omega'$ and the Reeb vector fields of  both cosymplectic structures coincide. Moreover, $\gamma= \eta-\eta'$ is a closed $1$-form and $i_{\mathcal R}\gamma=1-1=0.$ Conversely, if $(\omega,\eta)$ is a cosymplectic structure on $M$ with Reeb vector field ${\mathcal R}$ and $\gamma\in \Omega^1(M)$ is a 
closed $1$-form belonging  to $\langle{\mathcal R}\rangle^0,$ then $(\omega,\eta+\gamma)$ is a cosymplectic structure with Reeb vector field ${\mathcal R}$, that is, it induces the same mechanical presymplectic structure $(\omega,{\mathcal R})$ that the one of $(\omega,\eta)$ on $M$.

On the other hand, a contact structure on a manifold $M$ of odd dimension $2n+1$ is a 1-form $\Theta$ on $M$ such that $\Theta \wedge (d\Theta )^n$ is a volume form on $M$. If $\Theta$ is a contact structure then there exists a unique vector field $\mathcal{R}$ on $M$ such that
\[
i_{\mathcal{R}}\Theta =1, \qquad i_{\mathcal{R}}d\Theta =0,
\]
(for more details, see \cite{AM, Ar, LM}). It is clear that the couple $(d\Theta ,\mathcal{R})$ is a mechanical presymplectic structure. So, from a contact structure on $M$, one may define a mechanical presymplectic  structure.

Now, a typical example of a compact contact manifold is the sphere $S^{2n+1}$ of odd dimension $2n+1$ ($n\geq 1$). In fact, if $\lambda$ is the Liouville 1-form of $T^*\R ^{n+1}$
\[
\lambda = \sum _{i=1}^{n+1} p_i dq^i 
\]
with $(q^1,\ldots ,q^{n+1},p_1,\ldots ,p_{n+1})$ the canonical coordinates on $T^*\R^{n+1}$ and $i\colon S^{2n+1}\to T^*\R^{n+1}$ is the canonical inclusion then, it is well known that $\Theta =i^* \lambda$ is a contact structure on $S^{2n+1}$ (see, for instance, \cite{Bl}). Thus, $S^{2n+1}$ admits mechanical presymplectic structures. However, the Betti numbers of $S^{2n+1}$ satisfy 
\[ 
b_k(S^{2n+1})=0, \mbox{ for }k \in \{ 1,\ldots , 2n\}.
\]
Therefore, $S^{2n+1}$ doesn't admit cosymplectic structures.

\begin{remark}
Cosymplectic and contact structures  are examples of stable Hamiltonian structures \cite{HoZe} (see also \cite{Ac}). A stable Hamiltonian structure on a manifold $M$ of odd dimension $2n+1$ is a couple $(\omega,\lambda)$, where $\omega$ is a closed $2$-form, $\lambda$ is a  $1$-form, $\lambda\wedge \omega^n$ is a volume form and the Reeb vector field of $(\omega,\lambda)$ satisfies $i_{\mathcal R}d\lambda=0$. Then, it is clear that a stable Hamiltonian structure defines a mechanical presymplectic structure. However, for our interests, we don't need to fix the $1$-form $\lambda.$ So, we will deal with mechanical presymplectic structures. We also remark that coisotropic reduction for stable Hamiltonian structures was developed in \cite{LeIz}. 
\end{remark}

In what follows,  we will give a Marsden-Weinstein reduction theorem for mechanical presymplectic structures. Firstly, we will define the kind of symmetries considered. 
\begin{definition}\label{D}
A symmetry for a mechanical presymplectic structure $(\omega ,\mathcal{R})$ on a manifold $M$ is an action $\phi:G\times M\to M$ of a Lie group $G$ on $M$ such that 
 \begin{enumerate}
 \item[(i)] $\phi_g^*(\omega)=\omega$ and $T\phi_g({\mathcal R})={\mathcal R}\circ \phi_g$, for all $g\in G.$
 \item[(ii)] ${\mathcal R}(x)\notin T_x(G\cdot x),$ for all $x\in M.$ 
 \end{enumerate}
 \end{definition}
 
\begin{remark}
Condition (i) in Definition \ref{D} is natural, since it is just the invariance of the mechanical presymplectic structure under the action of the symmetry Lie group. 

On the other hand, condition (ii) is also natural. In fact, below we will use the symmetry to perform reduction of the mechanical presymplectic structure (see Theorem \ref{thpr}). However, if there were a point $x\in M$ such that $\mathcal{R} (x)\in T_x(G\cdot x)$ then the projection of $\mathcal{R}(x)$ to an appropriate orbit space would be zero and, therefore, such a projection could not be the Reeb vector field at the corresponding point of the reduced space.
\end{remark}

Note that the previous condition (i) implies that 
\begin{equation}\label{J}
 {\mathcal L}_{\xi_M}\omega=0
\end{equation}
 and therefore,  $i_{\xi_M}\omega$ is a closed $1$-form. This fact motivates the following notion.
\begin{definition}
A symmetry $\phi:G\times M\to M$ on a  mechanical presymplectic manifold  $(M,\omega,{\mathcal R})$   is a Hamiltonian presymplectic action  if there is a smooth map $J\colon M\to  {\mathfrak g}^*$ such that 
  \begin{equation}\label{Je2}i_{\xi_M}\omega=dJ_\xi, \mbox{ for all }\xi \in {\mathfrak g}.\end{equation} 
\end{definition} 
This kind of actions is a particular case of strongly presymplectic or Hamiltonian actions defined in \cite[Definition 1]{EMR2}.

As a consequence of Definition \ref{D}, we obtain a version of Noether theorem in this setting.
\begin{corollary}\label{cor:4.3'}
 Let $\phi:G\times M\to M$ be a Hamiltonian action for a mechanical presymplectic structure  $(\omega,{\mathcal R})$ on $M$ with momentum map $J:M\to {\mathfrak g}^*$. Then, for each $\xi \in \mathfrak{g}$, the smooth function $J_\xi$ is a first integral of the dynamical Reeb vector field $\mathcal{R}$.
\end{corollary}
\begin{proof}
Using Definition \ref{D}, it follows that 
\[
dJ_\xi (\mathcal{R})=(i_{\xi_M}\omega )(\mathcal{R})=-(i_\mathcal{R}\omega )(\xi _M)=0.
\]
\end{proof}
Next, we will prove that for Hamiltonian presymplectic actions which are infinitesimally free, the momentum map is a submersion. 
\begin{proposition}
Let $\phi:G\times M\to M$ be a Hamiltonian presymplectic action of a Lie group $G$ on the mechanical presymplectic manifold  $(M,\omega,{\mathcal R})$  with momentum map $J:M\to {\mathfrak g}^*$. Suppose that $\phi$ is infinitesimally free, i.e.
$$\xi_M=0 \mbox{ if and only if } \xi=0$$
for $\xi\in{\mathfrak g}.$ Then $J$ is a submersion and therefore, for all $\mu\in {\mathfrak g}^*$, $J^{-1}(\mu)$ is a submanifold of dimension $\dim M-\dim G.$ 
 \end{proposition}
\begin{proof}
We will prove that for all $x\in M,$ 
the dual map 
$T^*_xJ:T^*_{J(x)}{\mathfrak g }^*\cong {\mathfrak g }\to T^*_xM$ of $T_xJ\colon T_xM\to T_{J(x)}\mathfrak{g}^*\cong\mathfrak{g}^*$ is injective. In fact, 
if $\xi\in \ker T^*_xJ$ then 
$$0=T^*_xJ(\xi)=dJ_\xi(x).$$
Using (\ref{Je2}), we have that ${\xi_M(x)}\in \ker \omega(x)$, that is, there is a real value $\lambda_\xi^x\in {\mathbb R}$ such that $\xi_M(x)=\lambda_\xi^x{\mathcal R}(x).$ From condition (ii) of Definition \ref{D}, we deduce that $\lambda_\xi^x=0$ and therefore $\xi_M(x)=0,$ for all $x\in M.$ Since $\phi$ is infinitesimally free, then $\xi=0$. Thus, $\ker T^*_xJ=\{0\}.$
 \end{proof}
 In the rest of the section, we will prove a redution theorem for a Hamiltonian presymplectic action. For this purpose, we will use an algebraic result (Lemma 
\ref{Lema:4.5}), whose proof is contained in the appendix, and a description of the induced structure on the level sets of the momentum map (see Proposition \ref{prop:4.6} below).
 
\begin{lemma}\label{Lema:4.5}
 Let $V$ be a vector space of dimension $2n+1$, $\omega$ a $2$-form on $V$ and $R\in V$ such that
\[
\ker\omega =\langle R\rangle.
\]
Suppose that $A$ is a subspace of $V$  and $A^\perp$ denotes the subspace of $V$ 
 $$A^\perp=\{v\in V/\omega(v,a)=0,\mbox{ for all }a\in A\}.$$ 
We have these two alternatives 
 \begin{enumerate}
 \item[(i)] If $R\notin A$ then  $\dim A^\perp=\dim V-\dim A $
 \item[(ii)] If $R\in A$ then $\dim A^\perp=\dim V-\dim A+1$
 \end{enumerate}
 As a consequence $(A^\perp)^\perp=A \oplus \langle R\rangle$ if $R\notin A$  and $(A^\perp)^\perp=A$ if $R\in A.$ 
\end{lemma}
  
 \begin{proposition}\label{prop:4.6}
 Let $\phi:G\times M\to M$ be a Hamiltonian presymplectic action of a Lie group $G$ for a mechanical presymplectic structure  $(\omega,{\mathcal R})$  on $M$ with momentum map $J:M\to {\mathfrak g}^*.$ Suppose that $\phi$ is infinitesimally free. Let  $\mu\in \g^*$ with $J^{-1}(\mu)\not=\emptyset.$ 
 \begin{enumerate}
 \item[(i)] Then, 
 $$(T_x(J^{-1}(\mu)))^\perp:=\{ v\in T_xM/\omega(x)(v,u)=0 \mbox{ for all }u\in T_x(J^{-1}(\mu))\}=T_x(G\cdot x)\oplus \langle {\mathcal R}(x)\rangle,$$
 for all $x\in J^{-1}(\mu)$. 
 
 \item[(ii)] Suppose that $J$ is  $G$-equivariant with respect to the action $\phi$ and the coadjoint representation $Ad^*:G\times {\mathfrak g}^*\to {\mathfrak g}^*,$ i.e. 
\[
J\circ \phi_g=Ad^*_g\circ J \mbox{ for all } g\in G.
\] 
Then, 
\begin{enumerate}
\item[(a)] $\phi$ induces an  infinitesimally free action  of the isotropy subgroup $G_\mu$  on $J^{-1}(\mu).$
\item[(b)] If $\iota_\mu:J^{-1}(\mu)\to M$ is the canonical inclusion and $x\in J^{-1}(\mu),$ it follows that 
\[
\ker[(\iota_\mu^*\omega)(x)]=(T_x(J^{-1}(\mu))\cap T_x(J^{-1}(\mu)))^\perp=T_x(G_\mu\cdot x)\oplus  \langle {\mathcal R}(x)\rangle
 \]
 \end{enumerate}
 \end{enumerate}
 \end{proposition}
 \begin{proof} 
 (i) Firstly, we will prove that 
 \begin{equation}\label{J-1}
 T_x(J^{-1}(\mu))=(T_x(G\cdot x))^\perp =\{ u\in T_xM \, | \, \omega (x) (u ,\xi _M(x))=0, \mbox{ for }\xi \in \mathfrak{g} \},
 \end{equation}
 for all $x\in J^{-1}(\mu).$ In fact, using (\ref{Je2}), we deduce that $u\in T_x(J^{-1}(\mu))=\ker T_xJ$ if and only if
\[
0= \langle (T_xJ)(u), \xi \rangle = \langle dJ_\xi (x), u\rangle = \omega(x)(\xi_M(x),u),\mbox{ for }\xi\in {\mathfrak g}.
\]
So,  \eqref{J-1} holds.
 
Now, we apply the last part of the previous lemma 
\[ 
(T_x(J^{-1}(\mu)))^\perp=((T_x(G\cdot x))^\perp)^\perp=T_x(G\cdot x)\oplus \langle {\mathcal R}(x)\rangle.
\]
Here, we have used that ${\mathcal R}(x)\notin T_x(G\cdot x).$ 
 
(ii) If $g\in G_\mu$ then, from the equivariance of $J$, we have that, for all $x\in J^{-1}(\mu),$ 
$$J\circ \phi_g(x)=Ad^*_g(J(x))=Ad^*_g(\mu)=\mu.$$
Therefore, $\phi$ induces an infinitesimally free action of the isotropy subgroup $G_\mu$  on $J^{-1}(\mu)$. Moreover,  applying (i), we have that, for all $x\in J^{-1}(\mu),$ 
\[
\ker[(\iota_\mu ^*\omega)(x)]=\left ( T_x(J^{-1}(\mu)\right )\cap \left ( T_x(J^{-1}(\mu))\right )^\perp=\left ( T_x(J^{-1}(\mu)\right )\cap (T_x(G\cdot x)\oplus  \langle {\mathcal R}(x)\rangle)=T_x(G_\mu\cdot x)\oplus  \langle {\mathcal R}(x)\rangle.
\]
In the last equality, we have used that $\mathcal{R}(x)\in T_x(J^{-1}(\mu ))$ (see Corollary \ref{cor:4.3'}) and the fact that 
\[
T_x(J^{-1}(\mu ))\cap T_x(G\cdot x)=T_x(G_\mu \cdot x).
\]
Indeed, if $\xi \in \mathfrak{g}$ then, since $J$ is $G$-equivariant, it follows that 
\[
(T_xJ)(\xi _M(x))=\xi _{\mathfrak{g}^*} (J(x))=\xi _{\mathfrak{g}^*} (\mu)
\]
and, consequently,
\[
\xi _M(x)\in T_x(J^{-1}(\mu ))=\ker T_xJ \iff \xi _{\mathfrak{g}^*}(\mu )=0 \iff \xi \in \mathfrak{g}_\mu \iff \xi _M(x)\in T_x(G_\mu \cdot x).
\]
\end{proof}
As we mentioned before, the previous results allow to prove a reduction theorem for a Hamiltonian presymplectic action on a manifold endowed with a mechanical presymplectic structure. 
 
\begin{theorem}\label{thpr}
 Let $\phi:G\times M\to M$ be a Hamiltonian presymplectic action of a Lie group $G$ for a mechanical presymplectic structure 
 $(\omega,{\mathcal R})$ on $M$ with $G$-equivariant momentum map $J:M\to {\mathfrak g}^*$. Moreover, let $\mu$ be  an element of ${\mathfrak g}^*$ such that $J^{-1}(\mu)\not=\emptyset$ and suppose  that $\phi$ is infinitesimally free and the space of orbits $M_\mu=J^{-1}(\mu)/G_\mu$ is a manifold such that the canonical projection  $\pi_\mu:J^{-1}(\mu)\to M_\mu$ is a submersion. Then
 \begin{enumerate}
 \item[(i)] There exists a unique closed $2$-form $\omega_\mu$ on $M_\mu$ such that 
 \begin{equation}\label{w1}
 \pi_\mu^*(\omega_\mu)=\iota_\mu^*\omega,
 \end{equation}
 where $\iota_\mu:J^{-1}(\mu)\to T^*M$ is the inclusion map. 
 
\item[(ii)] The vector field ${\mathcal R}$ restricts to $J^{-1}(\mu)$ and its restriction is $\pi_\mu$-projectable. Its projection is a vector field ${\mathcal R}_\mu$ such that 
$$\ker\omega_\mu=\langle {\mathcal R}_\mu\rangle.$$
Thus, $(\omega_\mu,{\mathcal R}_\mu)$ is a mechanical presymplectic structure on $M_\mu=J^{-1}(\mu )/G_\mu$. 
 \end{enumerate}
 \end{theorem}
 \begin{proof}
(i) We will see that the 2-form $i^*\omega _\mu$ on $J^{-1}(\mu )$ is $\pi_\mu$-basic, that is,
\begin{equation}\label{26'}
i_{\xi_{J^{-1}(\mu )}}(i^*_\mu \omega )=0, \qquad 
i_{\xi_{J^{-1}(\mu )}}(d(i^*_\mu \omega ))=0,\mbox{ for }\xi \in \mathfrak{g}_\mu.
\end{equation}
This would imply that there exists a unique 2-form $\omega _\mu$ on $M_\mu$ such that
\begin{equation}\label{26''}
 \pi_\mu^*(\omega_\mu)=\iota_\mu^*\omega.
 \end{equation}
Since $\omega$ is closed, it is sufficient to prove the first part of \eqref{26'}. Now, using \eqref{Je2}, it follows that
\[
i_{\xi_{J^{-1}(\mu )}}(i^*_\mu \omega )=i^*_\mu(dJ_\xi )=0.
\]
In addition, using that $i^*_\mu\omega$ is a closed 2-form and the fact that $\pi^*_\mu$ is an injective morphism, we conclude that $\omega_\mu$ is closed.

(ii) From Corollary \ref{cor:4.3'}, we deduce that the restriction of $\mathcal{R}$ to $J^{-1}(\mu )$ is tangent to $J^{-1}(\mu )$. In addition, as $\mathcal{R}$ is $G$-invariant (see Definition \ref{D}), we have that $\mathcal{R}_{|J^{-1}(\mu )}$ is $G_\mu$-invariant and, therefore, $\pi_\mu$-projectable. So, there is a vector field ${\mathcal R}_\mu$ on $M_\mu$ such that  
\begin{equation}\label{RM}
{\mathcal R}_\mu(\pi_\mu(x))=T_x\pi_\mu({\mathcal R}(x)), \mbox{ for all }x\in J^{-1}(\mu).
\end{equation}
Now, using (\ref{w1}) and (\ref{RM}),  we have that 
\[
\pi_\mu^*(i_{{\mathcal R}_\mu}\omega_\mu)=\iota_\mu^*(i_{\mathcal R}\omega)=0.
\]
Thus, since $\pi_\mu^*$ is an injective morphism, we deduce that
\[
\mathcal{R}_\mu (\pi _\mu (x))\in \ker (\omega_\mu (\pi _\mu (x))),\mbox{ for }x\in J^{-1}(\mu ).
\]
Conversely, if $u\in T_x(J^{-1}(\mu ))$ and $(T_x\pi _\mu )(u)\in \ker (\omega_\mu (\pi_\mu (x)))$ then,
from \eqref{w1}, it follows that $u\in \ker (i^*_\mu \omega )(x)$. So, using Proposition \ref{prop:4.6}, we obtain that
\[
u=\xi_{J^{-1}(\mu )}(x)+\lambda \mathcal{R}(x), \mbox{ with }\xi\in \mathfrak{g}_\mu \mbox{ and } \lambda \in \R.
\]
Therefore,
\[
(T_x\pi_\mu )(u)=\lambda (T_x\pi _\mu )(\mathcal{R}(x))=\lambda \mathcal{R}_\mu (\pi _\mu (x))\in \langle \mathcal{R}_\mu(\pi _\mu (x))\rangle.
\]
This ends the proof of item (ii).
 \end{proof}
 
\begin{remark}\label{rmk:EMR2} In \cite{EMR2}, the authors consider  the reduction of a presymplectic Hamiltonian system on  manifolds where the presymplectic structure has arbitrary constant corank, showing a geometric reduction (Theorem 2) and a dynamic reduction (Theorem 4). In Theorem 2, it is obtained a reduced space with a presymplectic form that which has, under certain conditions, the same rank as the original one. Comparing both approaches, it can be seen that, under slightly more restrictive hypotheses than the ones assumed by us (in  \cite{EMR2}, it is required that the action is free and proper), (i) in Theorem \ref{thpr} is a consequence of their result for presymplectic structures of corank $1$. Since the proof is simple, we have decided to include it here to make the paper self-contained and facilitate the reader's understanding of the result.

However, the relevant differences lies in the dynamic result in \cite[Theorem 4]{EMR2}  and the second part of Theorem 4.1. In fact, in \cite{EMR2}, the authors obtain a reduced symplectic (not presymplectic!) Hamiltonian system. In order to deduce this, it is assumed that the kernel of the presymplectic structure is contained in the space generated by the infinitesimal generators of the action on $M$ (see  Assumption 2 in \cite{EMR2}). This condition is precisely the opposite to the one considered in our paper (see Definition \ref{D}).
\end{remark}

\section{Reduction of the evolution dynamics of a cosymplectic system}\label{sec:5}


To overcome the limitations of Albert's reduction Theorem stated at the end of Section \ref{sec:3}, we will modify the Hamiltonian cosymplectic structure $(M,\omega,\eta,J)$ to a new  one where the evolution vector field   $E^{(\omega,\eta)}_H$ is now  a Reeb vector field. The new momentum map determines first integrals for $E^{(\omega,\eta)}_H$ which are not in general first integrals for the  Reeb vector field or for the Hamiltonian vector field associated with  $(\omega,\eta)$ and $H$. 

\begin{proposition}\label{maJ}
Let $\phi:G\times (M,\omega,\eta) \to (M,\omega ,\eta)$ be a Hamiltonian action of a Lie group $G$ on the connected cosymplectic manifold $(M,\omega,\eta),$ with $G$-equivariant momentum map $J\colon M\to \g^*$. If $H:M\to \R$ is a $G$-invariant Hamiltonian function and $(\omega_H,\eta)$ is the cosymplectic structure defined in Proposition \ref{cc}, then $\phi:G\times (M,\omega_H,\eta) \to (M,\omega_H,\eta)$ is  a Hamiltonian  cosymplectic action   with  $G$-equivariant momentum map $J_H:M\to {\mathfrak g}^*$ defined by
\[
(J_H)_\xi=J_\xi-c_\eta(\xi)H,
\]
where $c_\eta$ is the cocycle defined by (\ref{cocycle}). 
\end{proposition}
\begin{proof}
From the $G$-invariance of $H$ and of the cosymplectic structure $(\omega,\eta)$, we deduce that $\phi:G\times (M,\omega_H,\eta) \to (M,\omega_H,\eta)$ is a cosymplectic action.

On the other hand, since  $\eta(\xi_M)=c_\eta (\xi)$, $H$ is invariant  and $\phi$ is Hamiltonian, then
\[
i_{\xi_M}\omega_H=i_{\xi_M}(\omega+dH\wedge \eta)= dJ_\xi-\eta(\xi_M)dH=d(J_\xi-c_\eta(\xi)H)= d (J_H)_\xi.
\]
Now, if $g\in G$ we have that 
\[
J_H\circ \phi_g=J\circ \phi_g -c_\eta (H\circ \phi_g)=Ad^*_g\circ J- H c_\eta.
\]
Here we have used the equivariance of $J$ and the fact that $H$ is $G$-invariant. So, we must see that $Ad^*_g(c_\eta )=c_\eta$. For this purpose, we will use the following relation
\[ 
(Ad_g\xi)_M(x)=T_{{\phi}_{g^{-1}}(x)}\phi_g(\xi_M(\phi_{g^{-1}(x)})),\mbox{ for }g\in G \mbox{ and } x\in M.
\]
Thus, since $\eta$ is $G$-invariant,
\[
\eta(x)((Ad_g(\xi))_M(x))=\eta(x)(T_{{\phi}_{g^{-1}}(x)}\phi_g(\xi_M(\phi_{g^{-1}}(x))))=\eta(\phi_{g^{-1}}(x))(\xi_M(\phi_{g^{-1}}(x)))
\]
 that is, $$Ad^*_g(c_\eta)(\xi)=\eta((Ad_g\xi)_M)=\eta(\xi_M)\circ \phi_{g^{-1}}=c_\eta(\xi).$$

In this last equality we have used that $\eta(\xi_M)$ is constant for each $\xi\in {\mathfrak g}.$

\end{proof}

The following result gives an answer about  how to reduce the evolution dynamics of a cosymplectic Hamiltonian system. The idea is to transform the evolution dynamics  into Reeb dynamics by modifying the cosymplectic structure. With this modification, we obtain that the information of the Hamiltonian function  is encoded within this new cosymplectic structure. In a second step,  one proves that, if we do an adequate modification of the momentum map,   the cosymplectic action with this new  momentum map  is Hamiltonian with respect to the new cosymplectic structure. However, in general, this Hamiltonian cosymplectic action does not satisfy  the Albert conditions. This reason justifies the use of the reduction of mechanical presymplectic structures instead of the reduction of cosymplectic structures. Next, we will present this process in detail.

Let $(M,\omega,\eta)$ be a connected cosymplectic manifold and $H:M\to {\mathbb R}$ a Hamiltonian function. Note that, if ${\mathcal R}$  is the Reeb vector field of this cosymplectic structure, then  $\ker\omega=\langle {\mathcal R}\rangle$ and  therefore $(\omega,{\mathcal R})$ is a mechanical presymplectic structure on $M$.

Moreover, the evolution vector field $E^{(\omega ,\eta)}_H$ is the Reeb vector field of the cosymplectic structure $(\omega_H=\omega+dH\wedge \eta, \eta)$ (see Proposition \ref{cc}). Therefore, $(\omega_H, E^{(\omega ,\eta)}_H)$ also is a mechanical presymplectic structure on $M$.

Now, let $\phi:G\times (M,\omega,\eta) \to (M,\omega,\eta)$ be  a Hamiltonian  cosymplectic action  of a Lie group $G$ on the connected cosymplectic manifold $(M,\omega,\eta)$ with  equivariant momentum map $J:M\to {\mathfrak g}^*.$ If $H:M\to {\mathbb R}$ is $G$-invariant, then, using Proposition \ref{maJ}, we have that $\phi:G\times (M,\omega_H,\eta) \to (M,\omega_H,\eta)$ is  a Hamiltonian cosymplectic action   with  $G$-equivariant momentum map $J_H:M\to {\mathfrak g}^*$ defined by
\begin{equation}\label{JH}
J_H=J- H c_\eta. 
\end{equation}
Moreover,  $E^{(\omega ,\eta)}_H$ is invariant.  

If, additionally, $E^{(\omega ,\eta)}_H(x)\notin T_x(G\cdot x)$ for all $x\in M$, we have that  the action $\phi:G\times (M,\omega_H,E^{(\omega ,\eta)}_H) \to (M,\omega_H,E^{(\omega ,\eta)}_H)$ is a Hamiltonian presymplectic action, and thus we can apply Theorem \ref{thpr}. So, we obtain the main result of this section.
 
 \begin{theorem}
 Let $\phi:G\times M\to M$ be a Hamiltonian cosymplectic action of a Lie group $G$ on the connected cosymplectic manifold  $(M,\omega,\eta)$    with $G$-equivariant momentum map $J:M\to {\mathfrak g}^*$ and $H:M\to {\mathbb R}$ a $G$-invariant Hamiltonian function.
 
Suppose that $\phi$ is infinitesimally free,  $E^{(\omega ,\eta )}_H(x)\notin T_x(G\cdot x),$ for all $x\in M$, and $\mu$ is an element of ${\mathfrak g}^*$ such that $J^{-1}_H(\mu)\not=\emptyset.$
 
If  the space of orbits $M_\mu=J_H^{-1}(\mu)/G_\mu$ is a manifold such that the canonical projection $\pi_\mu:J_H^{-1}(\mu)\to M_\mu$ is a submersion, then
 \begin{enumerate}
 \item[(i)] $\phi$ is a Hamiltonian presymplectic  action with respect to the mechanical presymplectic structure $(\omega_H,E^{(\omega ,\eta)}_H)$ and with momentum map $J_H$ given by (\ref{JH}).
 \item[(ii)] There exists a unique closed $2$-form $(\omega_H)_\mu$ on $M_\mu$ such that 
\[ 
 \pi_\mu^*((\omega_H)_\mu)=\iota_\mu^*\omega_H,
 \]
 where $\iota_\mu:J_H^{-1}(\mu)\to M$ is the inclusion map. 
\item[(iii)] The vector field $E^{(\omega ,\eta )}_H$ restricts to $J_H^{-1}(\mu)$ and its restriction is $\pi_\mu$ projectable. Its projection is a vector field ${\mathcal R}_\mu$ such that 
\[
\ker\omega_\mu=\langle {\mathcal R}_\mu\rangle.
\]
Thus, $((\omega_H)_\mu,{\mathcal R}_\mu)$ is a reduced mechanical presymplectic structure on $M_\mu$.
\end{enumerate}
\label{th:reductionMPS}
 \end{theorem}

Note that, using similar ideas as the ones described in Remark \ref{rem:aff}, in Theorem \ref{thpr}, Proposition \ref{maJ} and in Theorem \ref{th:reductionMPS} the $G$-equivariance condition is not necessary.

The following diagram summarizes the reduction process of the evolution dinamics of the system $(M,\omega,\eta,H).$ 

 \bigskip

\hspace{0cm}
{\tiny
$\xymatrix{
*+[F]{\mbox{ \begin{tabular}{c}{\sc Evolution dynamics }\\[3pt]$(M,{\omega,\eta,{\mathcal R})}$ cosymplectic system\\[3pt]$H:M\to {\mathbb R}$ $G$-invariant Hamiltonian function\\[3pt]$E^{(\omega ,\eta )}_H$ evolution vector field \\[3pt] $\phi: G\times M\to M$ Hamiltonian cosymplectic action\ \\[3pt]$J:M \to \g^*$ equivariant momentum map \\[3pt] $E^{(\omega ,\eta )}_H(x)\notin T_x(G\cdot x)$ \end{tabular}}}
\ar[rrr]^{\kern-20pt\mbox{\begin{tabular}{c}Modify to \\Reeb dynamics\end{tabular}}}
&&& *+[F]{\mbox{ \begin{tabular}{c}{\sc Reeb dynamics }\\[3pt]$(M,\omega_H=\omega + dH\wedge \eta,\eta,E^{(\omega ,\eta )}_H)$ cosymplectic manifold  \\[3pt] $\phi:G\times M\to M$ Hamiltonian  action\\[3pt] $J_H: M \to {\mathfrak g}^*$ equivariant momentum map  \\[3pt] $E^{(\omega ,\eta )}_H\notin T_x(G\cdot x)$\\[3pt]  \end{tabular}}}
\ar[dd]^{\mbox{\begin{tabular}{c}$\mu\in \g*$ \\[5pt] Mechanical presymplectic reduction\end{tabular}}}\\\\
 &&& *+[F]{\mbox{ \begin{tabular}{c}{\sc Reduced mechanical presymplectic manifold }\\[3pt] $(J_H^{-1}(\mu)/G_\mu,\omega_\mu,  ({E}^{(\omega ,\eta )}_H)_\mu)$\end{tabular}}}}$}

\begin{remark}
Suppose that $(\omega, \eta)$ is a cosymplectic structure on a manifold $M$, $H: M \to {\mathbb R}$ is Hamiltonian function and $G$ is a Lie group which acts on $M$. So, we can consider the mechanical presymplectic structure $(\omega_H, {\mathcal R}_H)$, with 
\[
\omega_H = \omega + dH \wedge \eta
\]
and ${\mathcal R}_H$ the Reeb vector field of the cosymplectic structure $(\omega_H, \eta)$. Then, in order to apply the results in Section \ref{section4} on reduction, we only need that the $2$-form $\omega_H$ and the vector field ${\mathcal R}_H$ to be $G$-invariant and that ${\mathcal R}_H(x) \notin T_x(G\cdot x)$, for all $x \in M$. We remark that it is not necessary to assume that the cosymplectic structure $(\omega, \eta)$ and the Hamiltonian function $H$ are $G$-invariant (as in Theorem \ref{th:reductionMPS}).
\end{remark}

\section{Examples}\label{section:examples}

\subsection{\for{toc}{$N$-dimensional harmonic oscillator}\except{toc}{$N$-dimensional harmonic oscillator seen by an observer that moves with a constant velocity (continuing with Example \ref{example1})}}

Consider a classical particle of mass $m$ moving in a harmonic oscillator potential with frequency $\Omega$ in $N$ dimensions described  in Example \ref{example1}. We recall that in this example,  the Latin indexes run from $1$ to $N$ while Greek indexes run from $2$ to $N$. 

From  (\ref{xiM}), one deduces that 
the action $\phi$ is infinitesimally free and  $E^{(\omega,\eta)}_H(q^i,p_i,t)\in \langle 1_M(q^i,p_i,t)\rangle=\langle\frac{\partial}{\partial t}-v \frac{\partial}{\partial q^1}\rangle$ if and only if $(q^i,p_i,t)$ belongs to the closed submanifold 
$$C=\{ (q^i,p_i,t)\in T^*\R^N\times \R\big| \; q^1=-vt, q^\alpha=0, p_1=-mv, p_\alpha=0,\mbox{ with } t\in \R\}.$$

In what follows we consider that $M$ is the open submanifold $(T^*\R^N\times \R)-C.$ Now, we look for the momentum map for the cosymplectic action $\phi$ with respect to $(\omega,\eta)$. Let  $J :M \to \g^*$ be 
the equivariant momentum map  $J_\xi (q^i,p_i,t)= - \xi v p_1$ for the cosymplectic structure $(\omega,\eta)$ described in Example \ref{example1}. 

We recall from \eqref{eq:J_Ha} that the modified momentum map $J_H = J - H c_\eta$ is given by
\begin{equation}
    \begin{split}
        (J_H)_\xi (q^i,p_i,t)&=\xi\left( - v p_1 - \frac{1}{2m} \sum_{i=1}^N p_i^2 - \frac{1}{2} m \Omega^2 \left( (q^1+v t)^2 + \sum_{\alpha=2}^N (q^\alpha)^2 \right)\right) \\
        &=\xi\left(\frac{m v^2}{2} -  \frac{1}{2m} \left( (p_1 + m v)^2 + \sum_{\alpha=2}^N p_\alpha^2 \right) - \frac{1}{2} m \Omega^2 \left( (q^1+v t)^2 + \sum_{\alpha=2}^N (q^\alpha)^2 \right)\right).
    \end{split}
\end{equation}
Now, we have the following cases: 
\begin{itemize}
\item 
For any $\mu < \displaystyle\frac{m v^2}{2}$, the preimage $J^{-1}_H (\mu)$ is the $2N$-dimensional submanifold given by
$$
    J^{-1}_H (\mu) = \left\{ \begin{array}{rcl}(q^i, p_i,t) \in M \; \bigg| \; &&\displaystyle \frac{1}{2m} \bigg( (p_1 + m v)^2 + \displaystyle\sum_{\alpha=2}^N p_\alpha^2 \bigg) + \displaystyle\frac{1}{2} m \Omega^2 \bigg( (q^1+v t)^2 + \displaystyle\sum_{\alpha=2}^N (q^\alpha)^2 \bigg) \\[5pt]&& = \displaystyle\frac{m v^2}{2} - \mu \end{array}\right\}.
$$

\item For any $\mu \geq \displaystyle\frac{m v^2}{2}$, we have that $J^{-1}_H (\mu) =\emptyset$.
\end{itemize}

When $\mu < \displaystyle\frac{m v^2}{2},$ to describe the quotient $J^{-1}_H (\mu)/\R$, let us define $\N=J^{-1}_H (\mu) \cap \varphi^{-1}(0)$, where $\varphi$ is the first projection, that is $\varphi(q^i,p_i,t) = q^1$. In fact, the closed manifold  $\N$ is the $2N$-dimensional ellipsoid given by
\begin{equation*}
\begin{split}
    \N &= \left\{ (q^\alpha, p_1, p_\alpha,t) \in \R^{2N}-\{(0,-mv,0,0)\} \; \bigg| \; \frac{1}{2m} \bigg( (p_1 + m v)^2 + \sum_{\alpha=2}^N p_\alpha^2 \bigg) + \frac{1}{2} m \Omega^2 \bigg( (v t)^2 + \sum_{\alpha=2}^N (q^\alpha)^2 \bigg) = \frac{m v^2}{2} - \mu \right\} \\
    &= \left\{ (q^\alpha, p_1, p_\alpha,t) \in \R^{2N} \; \bigg| \; \frac{1}{2m} \bigg( (p_1 + m v)^2 + \sum_{\alpha=2}^N p_\alpha^2 \bigg) + \frac{1}{2} m \Omega^2 \bigg( (v t)^2 + \sum_{\alpha=2}^N (q^\alpha)^2 \bigg) = \frac{m v^2}{2} - \mu \right\}.
    \end{split}
\end{equation*}
Deriving the constraint that defines ${\mathcal N},$ we deduce that the points of this submanifold satisfy 
\begin{equation}
    i^*\bigg(v d p_1 + \frac{1}{m} \sum_{i=1}^N p_i d p_i + m \Omega^2 \bigg( v^2 t dt + \sum_{\alpha=2}^N q^\alpha d q^\alpha \bigg) \bigg)= 0 ,
    \label{eq:tanN}
\end{equation}
where $i : M \to \N$ is the inclusion map.

From this equation we deduce that  the  map 
\begin{equation}\label{diff}
    \begin{split}
        \psi : \mathbb R \times \N &\to J^{-1}_H (\mu) \\
        (\hat t,q^\alpha,p_i,t) &\to (-v \hat t,q^\alpha,p_i,t+\hat t)
    \end{split}
\end{equation}
is well defined and, in fact, it is a global diffeomorphism. 

Now,  we have that  this diffeomorphism is equivariant, i.e. this diagram is commutative  

$$
\xymatrix{
\R\times \N  \ar[rr]^{\psi}\ar[d]^{\Phi_s}&& J_H^{-1}(\mu)\ar[d]^{\phi_s}
\\
\R\times \N\ar[rr]^{\psi} && J_H^{-1}(\mu) 
}
$$
where $\Phi:\R\times (\R\times \N)\to \R\times \N$ is the action $\Phi(s,(\hat t,q^\alpha,p_i,t))=(s+\hat t,q^\alpha,p_i,t).$ 

Therefore, we have that $J_H^{-1}(\mu)/\R\cong {\mathcal N}.$

Using the global diffeomorphism given in (\ref{diff}), we can pull-back $i^*_\mu\omega_H \in \Omega^2(J^{-1}_H (\mu))$ to $\psi^* i^*_\mu(\omega_H) \in \Omega^2(\mathbb R \times \N)$, which restricts to $\N. $ Thus, the corresponding $2$-form on $\R\times \N$ is (see (\ref{w}))
\begin{equation*}
\begin{split}
     \sum_{\alpha=2}^N d q^\alpha \wedge d p_\alpha + \frac{1}{m} \sum_{i=1}^N p_i d p_i \wedge d t + m \Omega^2 \sum_{\alpha=2}^N q^\alpha d q^\alpha \wedge d t 
    &+ \bigg( v d p_1 + \frac{1}{m} \sum_{i=1}^N p_i d p_i + m \Omega^2 \big( \sum_{\alpha=2}^N q^\alpha d q^\alpha + v^2 t d t \big) \bigg) \wedge d \hat{t} ,
\end{split}
\end{equation*}
but from \eqref{eq:tanN} the last term identically vanishes on the submanifold $\N$, and therefore, the reduced $2$-form on $\N$ is 
\begin{equation}
    (\omega_H)_\mu = \sum_{\alpha=2}^N d q^\alpha \wedge d p_\alpha + \frac{1}{m} \sum_{i=1}^N p_i d p_i \wedge d t + m \Omega^2 \sum_{\alpha=2}^N q^\alpha d q^\alpha \wedge d t .
\end{equation}
Finally, the restriction of evolution vector field $E^{(\omega,\eta)}_H$ to $J_H^{-1}(\mu)$ can be seen on $\R\times \N$ as 
\begin{equation*}
    E^{(\omega,\eta)}_H=\frac{1}{m} \bigg( \sum_{\alpha=2}^N p_\alpha \frac{\partial}{\partial q^\alpha} - \frac{p_1}{v}\frac{\partial}{\partial \hat t} \bigg) - m \Omega^2 \bigg( v t \frac{\partial}{\partial p_1} + \sum_{\alpha=2}^N q^\alpha \frac{\partial}{\partial p_\alpha}\bigg) + (\frac{p_1}{mv}+1)\frac{\partial}{\partial t} .
\end{equation*}
which projects on the reduced space $\N$ in the reduced vector field 
\begin{equation*}
    (E^{(\omega,\eta)}_H)_\mu=\frac{1}{m} \sum_{\alpha=2}^N p_\alpha \frac{\partial}{\partial q^\alpha} - m \Omega^2 \bigg( v t \frac{\partial}{\partial p_1} + \sum_{\alpha=2}^N q^\alpha \frac{\partial}{\partial p_\alpha}\bigg) + (\frac{p_1}{mv}+1)\frac{\partial}{\partial t} .
\end{equation*}

Thus, $(\N,(\omega_H)_\mu,(E^{(\omega,\eta)}_H)_\mu)$ is the reduced mechanical presymplectic manifold.

\subsection{Perturbation by a linearly polarized monocromatic plane wave}
\label{sec:planewave}

Consider an electron described by the time-independent Hamiltonian function
\begin{equation}
H_0 ({\mathbf q}, {\mathbf p}, t) = \frac{\mathbf p^2}{2 m} + V_0(\mathbf q) ,
\end{equation}
where $\mathbf q = (q^1,q^2,q^3)$ and $\mathbf p = (p_1,p_2,p_3)$, and $\mathbf p^2$ is shorthand for $\mathbf p \cdot \mathbf p$. Its coupling with an external electromagnetic field is given by 
\begin{equation}
H ({\mathbf q}, {\mathbf p}, t) = \frac{(\mathbf p- e \mathbf A(\mathbf q,t) )^2}{2 m} + e \varphi (\mathbf q,t) + V_0(\mathbf q) ,
\end{equation}
where $\mathbf A$ is the vector potential and $\varphi$ the scalar potential. Recall that the electric and magnetic fields are given by
\begin{equation}
\begin{split}
\mathbf E &= - \nabla \varphi - \frac{\partial \mathbf A}{\partial t} \\
\mathbf B &= \nabla \times \mathbf A \\
\end{split} .
\end{equation}

An interesting example is provided by a free electron, so $V_0(\mathbf q) = 0$, which is perturbed by a linearly polarized monocromatic plane wave (for instance a laser), \emph{i.e.}
\begin{equation}
\varphi = 0, \qquad\qquad  \mathbf A = A_0 \boldsymbol{\epsilon} \cos (\mathbf k \cdot \mathbf q - \omega t) \\
\end{equation}
where $\mathbf k$ is the wave--vector (recall that $\omega = | \mathbf k | c$) and $\boldsymbol \epsilon$ is the polarization vector (with $\boldsymbol \epsilon^2=1$).  Therefore, we can write
\begin{equation*}
H({\mathbf q}, {\mathbf p}, t) = H_0({\mathbf q}, {\mathbf p}) + H_1 ({\mathbf q}, {\mathbf p}, t) + H_2 ({\mathbf q}, {\mathbf p}, t),
\end{equation*}
with
\begin{equation*}
H_1 ({\mathbf q}, {\mathbf p}, t) = - \frac{e A_0}{m} \boldsymbol \epsilon \cdot \mathbf p \cos (\mathbf k \cdot \mathbf q - \omega t)
\end{equation*}
and
\begin{equation}
H_2 ({\mathbf q}, {\mathbf p}, t) =  \frac{e^2 A_0^2}{2 m} \cos^2 (\mathbf k \cdot \mathbf q - \omega t)
\end{equation}

If we assume that the perturbation is small ($A_0 \ll 1$), which is typically the case, we can neglect $H_2 (\mathbf q, \mathbf p, t)$.
Without any loss of generality, we can fix $\mathbf k = \mathbf u_1$ and $\boldsymbol \epsilon = \mathbf u_2$,  and thus we have that the Hamiltonian function of the system is given by
 
\begin{equation}
H(q^i, p_i, t) = \frac{1}{2 m} \sum_{i=1}^3 p_i^2 - \frac{e A_0}{m} p_2 \cos (q^1 - c t),
\end{equation}
whose differential reads
\begin{equation}
d H (q^i,p_i,t) = \frac{1}{m} \sum_{i=1}^3 p_i d p_i - \frac{e A_0}{m} \left( \cos (q^1 - c t) d p_2 - p_2 \sin (q^1 - c t) d q^1 + c \sin (q^1 - c t) d t \right).
\end{equation}
On $T^*\R^3\times \R\cong \R^7$ we consider the standard cosymplectic structure given in (\ref{eq:local}).

Then  the evolution vector field is
\begin{equation}\label{E2}
E^{(\omega,\eta)}_H (q^i,p_i,t) = \frac{1}{m} \sum_{i=1}^3 p_i \frac{\partial}{\partial q^i} - \frac{e A_0}{m} \cos (q^1 - c t) \frac{\partial}{\partial q^2} - \frac{e A_0}{m} p_2 \sin (q^1 - c t) \frac{\partial}{\partial p_1} + \frac{\partial}{\partial t} ,
\end{equation}\label{w2}
while the modified cosymplectic structure reads (see \eqref{wH})
\begin{equation*}
\left\{
\begin{array}{rcl}
\omega_H (q^i,p_i,t) &=& \displaystyle\sum_{i=1}^3 d q^i \wedge d p_i + \frac{1}{m} \sum_{i=1}^3 p_i d p_i \wedge d t - \frac{e A_0}{m} \left( \cos (q^1 - c t) d p_2 \wedge d t - p_2 \sin (q^1 - c t) dq^1 \wedge d t\right) \\
\eta&=&\displaystyle{dt}
\end{array}\right.
\end{equation*}

For this system, one consider the following cosymplectic symmetry   $\phi: \R \times (T^*\R^3\times \R) \to (T^*\R^3\times \R)$,  defined  by 
\begin{equation*}
    \phi (s, q^i,p_i,t) = (q^1+cs,q^\alpha,p_i,t + s) .
\end{equation*}
for the standard  cosymplectic structure  on $T^*\R^3\times \R.$ The fundamental vector field $1_M$ of $1\in \R$  associated to this action  is given by
\begin{equation*}
    1_M = c \frac{\partial}{\partial q^1} + \frac{\partial}{\partial t}.
\end{equation*}
Note that $\eta (1_M) =1$ and thus this action doesn't satisfy \eqref{cA}, that is, Albert's condition doesn't hold.  The integration of this vector field is 
\begin{equation*}
    \begin{split}
        q^1(s) &= c s + q^1 , \\
        q^\alpha (s) &= q^\alpha , \\
        p_i(s) &= p_i, \\
        t(s) &= s + t 
    \end{split}
\end{equation*}
with $\alpha\in \{2,3\}$ and $i\in \{1,2,3\}.$ A momentum map $J\colon T^*\R^N\times \R\to \R $ for $(\omega, \eta)$ satisfies that $i_{1_M} \omega = c  d p_1 = d J_1$, so we can take  $J(q^i,p_i,t) =   c p_1$. 

On the other hand, 
$E^{(\omega,\eta)}_H(q^i,p_i,t)\in \langle 1_M(q^i,p_i,t)\rangle$ if and only if $(q^i,p_i,t)$ belongs to the closed submanifold 
\begin{equation*}
    C = C_1 \cup C_2 \subset \R^7 ,
\end{equation*}
where
\begin{equation*}
    C_1 = \left\{ \left(ct+\frac{(2n-1)\pi}{2},q^2,q^3,mc,0,0,t \right)\; |\; t,q^\alpha\in \R,n\in \mathbb Z\right\}\subset \R^7 ,
\end{equation*}
and 
\begin{equation*}
    C_2 = \left\{\left(ct+n\pi,q^2,q^3,mc,(\--1)^{n}e A_0,0,t \right)\; |\; t,q^\alpha\in \R,n\in \mathbb Z\right\}\subset \R^7 .
\end{equation*}

Since $c_\eta(1)=\eta(1_M)=1$, the modified momentum map $J_H=J-H:(\R^7-C)\to \R$ on the open submanifold $\R^7-C$  is given by
\begin{equation*}
    J_H(q^i,p_i,t) = c p_1 - \frac{1}{2 m} \sum_{i=1}^3 p_i^2 + \frac{e A_0}{m} p_2 \cos (q^1 - c t).
\end{equation*}
The 6-dimensional level set of $J_H$ for $\mu\in \R$
\begin{equation*}
\begin{split}
    J_H^{-1} (\mu) &= \left\{ (q^i,p_i,t) \in \mathbb R^{7}-C \, | \,  c p_1 - \frac{1}{2 m} \sum_{i=1}^3 p_i^2 + \frac{e A_0}{m} p_2 \cos (q^1 - c t) = \mu \right\} \\
    &= \left\{ (q^i,p_i,t) \in \mathbb R^7-C \, | \, (p_1 - m c)^2 + (p_2 - e A_0 \cos (q^1 - c t))^2 + p_3^2 = m^2 c^2 + (e A_0)^2 \cos^2 (q^1 - c t) - 2 m \mu
    \right\} .
\end{split}
\end{equation*}
Note that none of the points in $C$ belong to $J_H^{-1} (\mu)$ so we can finally write
\begin{equation*}
    J_H^{-1} (\mu) = \left\{ (q^i,p_i,t) \in \mathbb R^7 \, | \, (p_1 - m c)^2 + (p_2 - e A_0 \cos (q^1 - c t))^2 + p_3^2 = m^2 c^2 + (e A_0)^2 \cos^2 (q^1 - c t) - 2 m \mu
    \right\} .
\end{equation*}

\begin{figure}[H]
\begin{center}
    \includegraphics[scale=0.6]{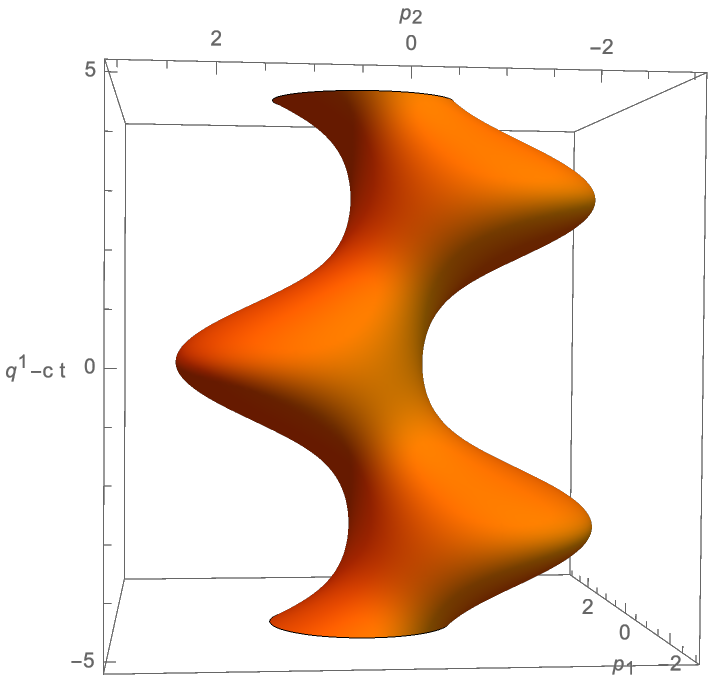}\hspace{0.5cm}\includegraphics[scale=0.6]{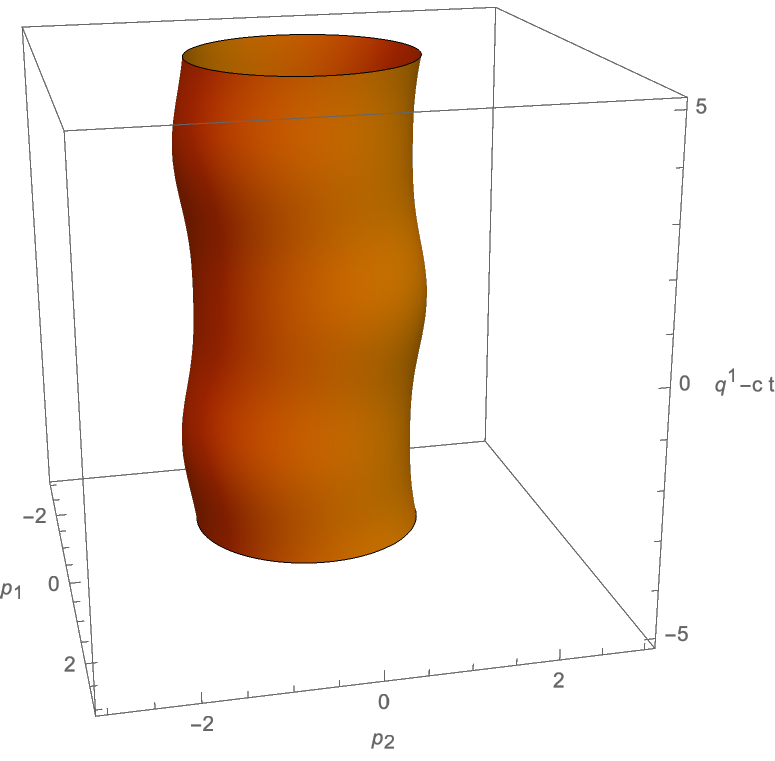}
    \caption{Cut with $p_3=0, q^2=0,q^3=0$ of the level sets $J_H^{-1} (0)$ for $e A_0 = 1$ (left) and $e A_0 = 0.1$ (right). We have used unit such that $m=c=1$.}
    \label{fig:}
    \end{center}
\end{figure}

Now, we consider the $5$-dimensional closed submanifold
\begin{equation}
    \N = J_H^{-1} (\mu) \cap \{ q^1 = 0\} \subset \R^7
   \end{equation}
given by 
\begin{equation}\label{const}
   \N =  \left\{ (q^i,p_i,t) \in \mathbb R^7 \, | \, 2mcp_1 - \sum_{i=1}^3 p_i^2 + 2{e A_0} p_2 \cos (c t) = 2m\mu
    \right\} .
\end{equation}
Deriving the previous constraint, we deduce that the points of this submanifold satisfy 
\begin{equation}
    i^*\bigg(2m \, c \, dp_1 - 2 \sum_{i=1}^3 p_i \, d p_i + 2 e A_0 ( \cos (c t) \, d p_2 - p_2 \sin (c t) \,c \, dt) \bigg)= 0 ,
\end{equation}
where $i : M \to \N$ is the inclusion map.

We can see that 
\begin{equation*}
    \begin{split}
        \Psi : \mathbb R \times \N &\to J^{-1}_H (\mu) \\
        (\hat t,q^\alpha,p_i,t) &\to (c \hat t,q^\alpha,p_i,t+\hat t)
    \end{split}
\end{equation*}
is a global equivariant diffeomorphism  when on $J_H^{-1}(\mu)$ we  consider the action $\phi$ and on $\R\times \N$ the action
$\Phi:\R\times (\R\times \N)\to \R\times \N$ given by  $\Phi(s,(\hat t,q^\alpha,p_i,t))=(s+\hat t,q^\alpha,p_i,t).$ Therefore, $J_H^{-1}(\mu)/\R$ can be identified with the manifold $\N.$ 

Using the derivation of the equation (\ref{const}), we deduce that $\Psi^*(\omega_H)$ restrict to $\N$ and therefore (see (\ref{w2}))
$$(\omega_H)_\mu=\sum_{\alpha=2}^3 dq^\alpha\wedge dp_\alpha + \left(\frac{1}{m}\sum_{i=1}^3p_idp_i -\frac{eA_0}{m}cos(ct)dp_2 \right)\wedge dt.$$

Finally, the restriction of evolution vector field $E^{(\omega,\eta)}_H$ to $J_H^{-1}(\mu)$ can be seen on $\R\times \N$ as (see (\ref{E2}))
\begin{equation*}
    E^{(\omega,\eta)}_H=\frac{1}{m} \bigg( \sum_{\alpha=2}^3 p_\alpha \frac{\partial}{\partial q^\alpha} + \frac{p_1}{c}\frac{\partial}{\partial \hat t} \bigg) -\frac{eA_0}{m}\cos(ct)\frac{\partial }{\partial q_2} + \frac{eA_0}{m}p_2\sin(ct)\frac{\partial} {\partial p_1}+ (1-\frac{p_1}{mc})\frac{\partial }{\partial t}.
\end{equation*}
which projects on the reduced space $\N$ in the reduced vector field 
\begin{equation*}
    (E^{(\omega,\eta)}_H)_\mu=\frac{1}{m}  \sum_{\alpha=2}^3 p_\alpha \frac{\partial}{\partial q^\alpha} -\frac{eA_0}{m}\cos(ct)\frac{\partial }{\partial q_2} + \frac{eA_0}{m}p_2\sin(ct)\frac{\partial} {\partial p_1}+ (1-\frac{p_1}{mc})\frac{\partial }{\partial t}.
\end{equation*}

Therefore, the reduced mechanical presymplectic manifold is  
$(\N,(\omega_H)_\mu,(E^{(\omega,\eta)}_H)_\mu).$ 

\appendix
\section{Proof of Lemma \ref{Lema:4.5}}
We consider the lineal map
 $$\flat:V\to V^*, \;\;\; \flat(v)=i_v\omega.$$ 
 Then $\dim \ker\flat=1$ and $\dim \flat(V)=\dim V-\dim\ker \flat=2n.$ Moreover, since $R\in A^\perp$, we obtain that 
 \[
 \flat(A^\perp)=\{i_v\omega/\omega(v,a)=0,\forall a\in A\}=A^o\cap \flat(V)\mbox{ and } \ker(\flat_{|A^\perp})=(\ker\flat)\cap A^\perp=\langle R\rangle,
 \]
 with $A^o$ the annihilator of $A$, i.e.
 $$A^o=\{\alpha\in V^*/\alpha(v)=0, \forall a\in A\}.$$
 
Thus, we deduce that
\begin{eqnarray*}
    \dim A^\perp&=&\dim(A^\perp \cap \ker\flat)+\dim (A^o\cap \flat(V))=1-\dim (A^o+ \flat(V))+\dim A^o + \dim \flat(V)\\
    &=&1-\dim (A^o+ \flat(V))+(2n+1-\dim A )+2n\\ &=&4n+2-\dim (A^o+ \flat(V))-\dim A.
\end{eqnarray*}
If $R\notin A,$  then $A^o\not\subset \flat(V).$ In fact, if $A^o\subset \flat(V), $ then $(\flat(V))^o\subset A$. But, $R\in (\flat(V))^o$ and $R\notin A$. Thus, 
\[
\dim A^\perp=4n +2-2n-1-\dim A=2n+1-\dim A=\dim V -\dim A.
\]
On the other hand, if $R\in A$, then $(\flat(V))^o\subset A$. In fact, let $u\in (\flat(V))^o$. Then $u\in \ker \flat=\langle R\rangle$ and therefore, $u\in A.$ In consequence, in this case,  $A^o\subset \flat(V)$  and 
\[
\dim A^\perp=4n+2-2n-\dim A=\dim V- \dim A+1.
\]
Now, if $R\notin A$, we have that $A\oplus \langle R\rangle \subseteq (A^\perp)^\perp$. Indeed, if $u\in A$ then $\omega(u,v)=0$ for all $v\in A^\perp,$ and $R\in (A^\perp)^\perp$ On the other hand, using  (i) and (ii) in this lemma 
\[
\dim (A^\perp)^\perp=\dim V- \dim A^\perp+1=\dim A + 1.
\] 
Therefore, $(A^\perp)^\perp=A \oplus \langle R\rangle.$ 

Finally, if $R\in A$ then $A\subseteq (A^\perp)^\perp$ and $\dim (A^\perp)^\perp =\dim A$. This implies that  $(A^\perp)^\perp=A.$

\section*{Acknowledgements}

The authors thank J de Lucas and B Zawora for useful comments related with the contents of this paper.
I. Gutierrez-Sagredo has been supported by the grant PID2023-148373NB-I00 funded by MCIN /AEI /10.13039 /501100011033 / FEDER, UE, and by the Q-CAYLE Project funded by the Regional Government of Castilla y Le\'on (Junta de Castilla y Le\'on) and by the Ministry of Science and Innovation (MCIN) through the European Union funds NextGenerationEU (PRTR C17.I1).  D. Iglesias, J.C. Marrero and E. Padr\'on acknowledge financial support from the Spanish Ministry of Science and Innovation under grant PID2022-137909NB-C22. All the authors  has been partially supported by Agencia Estatal de Investigaci\'on (Spain) under grant RED2022-134301-TD.  I. Gutierrez-Sagredo thanks Universidad de La Laguna, where part of the work has been done, for the hospitality and support.


\end{document}